%% file: main.tex
\DeclareSymbolFont{AMSb}{U}{msb}{m}{n}
\documentclass[11pt,a4paper,twoside,leqno,noamsfonts]{amsart}
           \makeatletter
           \newcommand{\mylabel}[2]{#2\def\@currentlabel{#2}\label{#1}}
           \renewcommand\@biblabel[1]{#1.}
           \makeatother
\usepackage[english]{babel}
\usepackage[dvipsnames]{xcolor}
\definecolor{britishracinggreen}{rgb}{0.0, 0.26, 0.15}
\definecolor{cobalt}{rgb}{0.0, 0.28, 0.67}
\usepackage[utopia]{mathdesign}
    \DeclareSymbolFont{usualmathcal}{OMS}{cmsy}{m}{n}
    \DeclareSymbolFontAlphabet{\mathcal}{usualmathcal}
\usepackage[a4paper,top=3cm,bottom=3cm,left=2.8cm,
           right=2.8cm,bindingoffset=5mm]{geometry}
\usepackage[utf8]{inputenc}
\usepackage{braket,caption,comment,mathtools,stmaryrd}
\usepackage{multirow,booktabs,microtype}
\usepackage[colorlinks,bookmarks]{hyperref} %
\hypersetup{colorlinks,%
            citecolor=britishracinggreen,%
            filecolor=black,%
            linkcolor=cobalt,%
            urlcolor=black}
\numberwithin{equation}{section}
\input{macros}

\title[Virtual classes and virtual motives of Quot schemes on 3-folds]{Virtual classes and virtual motives \\ of Quot schemes on threefolds}

\author[Andrea T. Ricolfi]{Andrea T. Ricolfi}
\address{SISSA, Via Bonomea 265 Trieste}
\email[Andrea T. Ricolfi]{aricolfi@sissa.it}


\begin{document}
\maketitle

\begin{abstract}
For a simple, rigid vector bundle $F$ on a Calabi--Yau $3$-fold $Y$, we construct a \emph{symmetric obstruction theory} on the Quot scheme $\Quot_Y(F,n)$, and we solve the associated enumerative theory. We discuss the case of other $3$-folds.
Exploiting the critical structure on the local model $\Quot_{\A^3}(\O^{\oplus r},n)$, we construct a \emph{virtual motive} (in the sense of Behrend--Bryan--Szendr\H{o}i) for $\Quot_Y(F,n)$ for an arbitrary vector bundle $F$ on a smooth $3$-fold $Y$. We compute the associated motivic partition function. 
We obtain new examples of higher rank (motivic) Donaldson--Thomas invariants.
\end{abstract}

{\hypersetup{linkcolor=black}
\tableofcontents}


\section{Introduction}
\subsection*{Overview}
The goal of this paper is to show that, for a locally free sheaf $F$ on a complex $3$-fold $Y$, the Quot scheme
\[
\Quot_Y(F,n) = \Set{F\onto Q \,\big|\, \dim(Q)=0,\,\chi(Q)=n}
\]
carries a degree $0$ \emph{virtual fundamental class} (under suitable assumptions), as constructed by Behrend--Fantechi \cite{BFinc}, as well as a \emph{virtual motive} in the sense of Behrend--Bryan--Szendr\H{o}i \cite{BBS}. Therefore enumerative and motivic invariants can be attached to $\Quot_Y(F,n)$. Our results yield new explicit examples of higher rank \emph{Donaldson--Thomas invariants} and higher rank \emph{motivic Donaldson--Thomas invariants} of Calabi--Yau $3$-folds.

Our first main result (proved in Theorem \ref{THM:Ob_Theory}) is the following.

\begin{thm}\label{thmA}
Let $Y$ be a smooth complex projective $3$-fold, $F$ a simple rigid vector bundle on $Y$. Then $\Quot_Y(F,n)$ admits a $0$-dimensional perfect obstruction theory in the following situations:
\begin{enumerate}
    \item $H^i(Y,\O_Y)=0$ for $i>0$ and $F$ is exceptional;\label{exceptional_case}
    \item $Y$ is Calabi--Yau.
\end{enumerate}
In the Calabi--Yau case, the obstruction theory is symmetric.
\end{thm}

Under the assumptions of Theorem \ref{thmA}, one can see $\Quot_Y(F,n)$ as a fine moduli space of simple sheaves (the \emph{kernels} of the surjections), and form the Donaldson--Thomas partition function
\be\label{series:DT}
\DDT_F(q) = \sum_{n\geq 0} q^n \left(\int_{[\Quot_Y(F,n)]^{\vir}}1\right)   \in \Z\llbracket q \rrbracket.
\ee
In the Calabi--Yau case, we deduce (cf.~Corollary \ref{cor:DT_CY3}) the identity
\[
\DDT_F(q) = \mathsf M((-1)^rq)^{r\chi(Y)},
\] 
where $\mathsf M(q) = \prod_{m\geq 1}(1-q^m)^{-m}$
is the MacMahon function and $r=\rk F$. We conjecture a general formula for $\DDT_F(q)$, in the case where $(Y,F)$ satisfies \eqref{exceptional_case}, in Section \ref{sec:DT_general_threefold}.

\smallbreak
To state our second main result, let us fix an \emph{arbitrary} smooth $3$-fold $Y$, and a vector bundle $F$ on $Y$ of rank $r$. Let $K_0(\Var_{\C})$ be the Grothendieck ring of complex varieties, and let $\L = [\A^1]$ be the Lefschetz motive. In Section \ref{Sec:Virtual_Motives} we define motivic weights
\[
\bigl[\Quot_Y(F,n)\bigr]_{\vir} \in \mathcal M_{\C} = K_0(\Var_{\C})\bigl[\L^{-\frac{1}{2}}\bigr]
\]
that are \emph{virtual motives} in the sense of \cite{BBS}, i.e.~their Euler characteristic computes the virtual Euler characteristic $\widetilde\chi(\Quot_Y(F,n)) = \chi(\Quot_Y(F,n),\nu)\in \Z$ defined by means of Behrend's microlocal function \cite{Beh}.
We express the generating function
\be\label{Motivic_Partition}
\mathsf Z_r(Y,t) = \sum_{n\geq 0}\,\bigl[\Quot_Y(F,n)\bigr]_{\vir}\cdot t^n
\ee
in terms of the \emph{motivic exponential} (reviewed in Section \ref{sec:Motivic_Exp}). The next result (proven in Theorem \ref{Thm:Motivic_Partition_Function}) recovers the calculation \cite[Thm.~4.3]{BBS} by Behrend--Bryan--Szendr\H{o}i for the Hilbert scheme of points $\Hilb^nY$ if one sets $r=1$.

\begin{thm}\label{thmB}
The motivic partition function \eqref{Motivic_Partition} satisfies 
\[
\mathsf Z_r(Y,(-1)^rt)=\Exp\left(
\frac{(-1)^rt\left[Y \times \P^{r-1}\right]_{\vir}}{\bigl(1-(-\mathbb L^{-\frac{1}{2}})^rt\bigr)\bigl(1-(-\mathbb L^{\frac{1}{2}})^rt\bigr) }\right).
\]
\end{thm}
If $F$ is a simple, rigid vector bundle on a Calabi--Yau $3$-fold $Y$, the coefficients of the series \eqref{Motivic_Partition} refine the enumerative Donaldson--Thomas invariants encoded in \eqref{series:DT}.
Thus Theorem \ref{thmB} explicitly computes generating functions of higher rank \emph{motivic Donaldson--Thomas invariants}. As an example, consider a stable arithmetically Cohen--Macaulay rank $2$ bundle $F$ on a general quintic $Y\subset \P^4$ (cf.~Example \ref{Ex:ACM_Quintic}). Then $F$ is rigid, and Theorem \ref{thmB} yields (up to a sign) a refinement of the enumerative formula
\[
\DDT_F(q) = \mathsf M(q)^{-400}.
\]

\subsubsection*{Cohomological DT theory}
It is proven in \cite[Thm.~2.6]{BR18} that $\Quot_{\A^3}(\O^{\oplus r},n)$ is the critical locus of a regular function $f_{r,n}$ for all $r$ and $n$ (cf.~Section \ref{subsec:local_model}). We observe, using one of the main results of \cite{Davison16}, that the compactly supported \emph{vanishing cycle cohomology} 
\be\label{VC_quot}
\mathrm H_c\bigl(\Quot_{\A^3}(\O^{\oplus r},n),\Phi_{f_{r,n}}\bigr)
\ee
is \emph{pure}, and of Tate type, for all $n$. Moreover, in Section \ref{sec:VC_Cohomology} we compute, for fixed $r\geq 1$, the generating function of Hodge polynomials of \eqref{VC_quot}, cf.~Formula \eqref{VC_calculation}.

\subsection*{Related work in the rank $1$ case}
The enumerative theory of $\Hilb^nY$ has been solved in \cite{BFHilb,LEPA,JLi}. The first breakthrough in motivic Donaldson--Thomas theory was the definition and explicit calculation of the virtual motive of $\Hilb^nY$ on a $3$-fold \cite{BBS}. 

Concerning Hilbert schemes of subschemes $Z\subset Y$ with $\dim Z\leq 1$ in a projective $3$-fold $Y$, the contribution of a smooth curve $C\subset Y$, embedded with ideal sheaf $\mathscr I_C$, is encoded in the Quot scheme
\[
\Quot_Y(\mathscr I_C,n)\subset \Hilb_{\chi(\O_C) + n}(Y,[C]).
\]
The $C$-local enumerative Donaldson--Thomas theory was solved in \cite{LocalDT,Ricolfi2018}, whereas the motivic side was studied by Davison and the author in \cite{DavisonR}.

\subsubsection*{Conventions}
All schemes are locally of finite type over $\C$. For a scheme $X$, by $\DD(X) = \DD(\QCoh(X))$ we denote its derived category, and we let $(-)^\vee = \RRlHom(-,\O_X)$ be the derived dualising functor. For a torsion free sheaf $E$ on a variety $Y$ we denote by $\Ext^i(E,E)_0$ the kernel of the trace map $\tr^i\colon \Ext^i(E,E) \to H^i(Y,\O_Y)$, see \cite[Section 10.1]{modulisheaves} for its construction. A locally free sheaf (or vector bundle) $F$ on a variety is called \emph{simple} if $\Hom(F,F)=\C$, \emph{rigid} if $\Ext^1(F,F)=0$,
\emph{exceptional} if it is simple and $\Ext^i(F,F)=0$ for all $i>0$. A \emph{Calabi--Yau $3$-fold} is a smooth projective variety $Y$ of dimension $3$, such that $\omega_Y \cong \O_Y$ and $H^1(Y,\O_Y)=0$.

\subsection*{Acknowledgements}

The author wishes to thank Ben Davison, Barbara Fantechi and Martijn Kool for very helpful discussions. We owe a debt of gratitude to Dragos Oprea for generously sharing his insights on the problem. Many thanks to the anonymous referee for spotting several inaccuracies and helping to improve the text. Finally, thanks to SISSA for the great working conditions offered during the completion of this project.

\section{Preliminaries}\label{Sec:Virtual_Stuff}

In this section we set the main tools that will be used throughout the paper.

\subsection{Obstruction theories}
We refer the reader to \cite{BFinc,BFHilb} for more details on obstruction theories and virtual classes. Here we only recall the main definitions.

Let $X$ be a finite type $\C$-scheme, and let $L_X^{\bullet} \in \DD^{(-\infty,0]}(X)$ be Illusie's cotangent complex. 

\begin{definition}[{\cite[Def.~4.4]{BFinc} and \cite[Def.~1.10]{BFHilb}}]\label{def:Ob_Theory}
An \emph{obstruction theory} on $X$ is a morphism $\phi\colon \mathbb E \to L_X^{\bullet}$ in $\DD(X)$ such that $h^0(\phi)$ is an isomorphism and $h^{-1}(\phi)$ is surjective. If $\mathbb E$ is perfect of perfect amplitude contained in $[-1,0]$, we say that $\phi$ is \emph{perfect}. If there exists an isomorphism $\theta\colon \mathbb E\,\,\widetilde{\to}\,\, \mathbb E^\vee[1]$ such that $\theta^\vee[1] = \theta$, we say that $\phi$ is \emph{symmetric}. The \emph{virtual dimension} of a perfect obstruction theory is the integer $\vd = \rk \mathbb E$, i.e.~the difference $\rk E^0-\rk E^{-1}$ if $\mathbb E$ is locally written $[E^{-1} \to E^0]$.
\end{definition}

Throughout, we let
\[
\L_X = \tau_{\geq -1}L_X^{\bullet}  \in \DD^{[-1,0]}(X)
\]
be the cut-off at $-1$ of the full cotangent complex. We will only treat \emph{perfect} obstruction theories, which can be viewed as morphisms $\phi\colon \mathbb E\to \mathbb L_X$.
If $X$ is embeddable in a smooth scheme $U$ with ideal sheaf $I\subset \O_U$, then one has a canonical isomorphism
\[
\L_X = \bigl[I/I^2 \xrightarrow{\dd} \Omega_U\big|_X\bigr]
\]
where $\dd$ is the exterior derivative.

\subsection{Rings of motives and structures on them}
Most of the conventions recalled here are taken verbatim from \cite[Section 1]{DavisonR}. We will need this material (only) in Section \ref{Sec:Virtual_Motives}, so the reader not interested in the motivic part of the paper can safely skip the rest of this section.

Let $S$ be a variety over $\mathbb C$, and let $K_0(\Var_S)$ be the Grothendieck ring of $S$-varieties. The ring of \emph{motivic weights} over $S$ is the ring
\[
\mathcal M_S = K_0(\Var_{S})\bigl[\L^{-\frac{1}{2}}\bigr]
\]
obtained by formally inverting a square root of the Lefschetz motive $\L=[\A^1_S]$.

A morphism of schemes $f\colon S\to T$ induces, by fibre product, a map of rings $f^*\colon \mathcal M_T \to \mathcal M_S$, while composition with $f$ gives an $\mathcal M_T$-linear direct image homomorphism $f_!\colon \mathcal M_S \to \mathcal M_T$. If $f\colon S\to \Spec \C$ is the structure morphism of $S$, we write $\int_S$ instead of $f_!$ . If $S$ and $S'$ are two varieties, the exterior product 
\[
\mathcal M_S\times \mathcal M_{S'}\xrightarrow{\boxtimes}\mathcal M_{S\times S'}
\]
is defined on generators of $K_0(\Var)$ by sending $(u,v)\mapsto u\times v$ and then extended by linearity.

\begin{definition}
We denote by $S_0(\Var_S) \subset K_0(\Var_S)$ the sub semigroup of effective motives, i.e.~the semigroup generated by classes $[X\to S]$ of complex quasi-projective $S$-varieties modulo the scissor relations. Its image in $\mathcal M_S$ is the sub semigroup $\mathcal M_S^{\eff}\subset \mathcal M_S$ consisting of sums of elements of the form
\[
(-\mathbb{L}^{\frac{1}{2}})^n[X\rightarrow S],\quad n\in \Z.
\]
\end{definition}

\subsubsection{Equivariant theory}\label{sec:equivariant_Kgroups}

Recall that if $S$ is a variety with a good action by a finite group $G$ (i.e.~such that every point of $S$ has an affine $G$-invariant open neighborhood), the quotient $S/G$ exists as a variety.

\begin{definition}\label{Equivariant_K_Group}
Let $G$ be a finite group, $S$ a variety with good $G$-action. We denote by $\widetilde{K}_0^{G}(\Var_S)$ the abelian group 
generated by isomorphism classes $[X\to S]$ of $G$-equivariant $S$-varieties
with good action, modulo the $G$-equivariant scissor relations.
We define the $G$-\emph{equivariant Grothendieck group} $K_0^{G}(\Var_S)$ by imposing the further relations $[V\to X \to S]= [\A^r_X]$,
whenever $V\to X$ is a $G$-equivariant vector bundle of rank $r$, with $X\to S$ a $G$-equivariant $S$-variety.
The element $[\A^r_X]$ in the right hand side is taken with the $G$-action induced by the trivial action on $\A^r$ and the isomorphism $\A^r_X=\A^r\times X$.
\end{definition}

There is a natural ring structure on $\widetilde{K}_0^{G}(\Var_S)$ given by taking the diagonal action on $X\times_SY$, for two equivariant $S$-varieties $X\to S$ and $Y\to S$. Inverting a square root of $\L$, one obtains the rings $\widetilde{\mathcal M}_S^G$ and $\mathcal M_S^G$ of $G$-\emph{equivariant motivic weights}. These rings fit in a commutative diagram 
\be\label{map:quot1map}
\begin{tikzcd}
\widetilde{K}_0^{G}(\Var_S) \arrow{r}{\pi_G} \arrow{d} & K_0(\Var_{S/G})\arrow{d}\\
\widetilde{\mathcal M}_S^G\arrow{r}{\pi_G} & \mathcal M_{S/G}
\end{tikzcd}
\ee
where the top map $\pi_G$ is defined on generators by taking the orbit space,
\[
[X\to S] \mapsto [X/G \to S/G],
\]
and the bottom map is the extension determined by $\mathbb{L}^{\frac{n}{2}}\cdot [X\rightarrow S]\mapsto \mathbb{L}^{\frac{n}{2}}\cdot \pi_G[X\rightarrow S]$.
The map $\pi_G$ does not always extend to $\mathcal M_S^G$. It does if $G$ acts freely on $S$.

\smallbreak
A special case of Definition \ref{Equivariant_K_Group} yields the \emph{monodromic ring of motivic weights}
\[
\mathcal M_S^{\hat\mu},
\]
where $\hat\mu = \varprojlim \mu_n$ is the procyclic group of roots of unity. We have an Euler characteristic homomorphism 
\be\label{eulerchi}
\chi\colon \mathcal M_{\C}^{\hat\mu}\to \Z,\qquad \chi(\L^{-\frac{1}{2}})=-1.
\ee

\subsubsection{Lambda ring structures}\label{sec:lambda_rings}

Let $n>0$ be an integer, and let $\mathfrak{S}_n$ be the symmetric group of $n$ elements. By \cite[Lemma~1.6]{DavisonR}, namely the relative version of \cite[Lemma~2.4]{BBS}, there exist ``$n$-th power'' maps fitting in a commutative diagram
\be\label{powermap}
\begin{tikzcd}
K_0(\Var_S)\arrow{r}{(\,\cdot\,)^{\otimes n}}\arrow{d} & \widetilde{K}_0^{\,\mathfrak S_n}(\Var_{S^n})\arrow{d} \\
\mathcal M_S \arrow{r}{(\,\cdot\,)^{\otimes n}} & \widetilde{\mathcal M}_{S^n}^{\,\mathfrak S_n}
\end{tikzcd}
\ee
where $S^n = S\times \cdots\times S$ carries the natural $\mathfrak{S}_n$-action.
For $A\in\mathcal{M}_{S}$, define 
\[
\prsigma^n(A)=\pi_{\mathfrak{S}_n}(A^{\otimes n}) \in \mathcal M_{S^n/\mathfrak S_n}.
\]
The \textit{lambda ring} operations on $\mathcal M_{\C}$ are defined by $\sigma^n(A)=\prsigma^n(A) \in \mathcal M_{\C}$ for $A$ effective, and then taking the unique extension to a lambda ring on $\mathcal M_{\C}$, determined by the relation
\begin{equation}
\label{lambda_rel}
\sum_{i=0}^n\sigma^i([X]-[Y])\sigma^{n-i}[Y]=\sigma^n[X].
\end{equation}
Note that $\sigma^n(-\L^{1/2})=(-\L^{1/2})^n$. 

If $S$ comes with a commutative associative map $\nu\colon S\times S\to S$, we likewise define 
\[
\prsigma_{\nu}^n(A)=\nu_!\prsigma^n(A),
\]
where we abuse notation by denoting by $\nu$ the map $S^n/\mathfrak{S}_n\rightarrow S$.  As above, using the analogue of the relation \eqref{lambda_rel}, there is a unique set of lambda ring operators $\sigma_{\nu}^n$ agreeing with $\prsigma_{\nu}$ on effective motives. 

As a special case, we can consider $(S,\nu)=(\mathbb N,+)$, viewed as a symmetric monoid in the category of schemes. We obtain operations $\prsigma^n$ and $\sigma^n$ on $\mathcal{M}_{\C}\llbracket t\rrbracket$ via the isomorphism
\be\label{Iso_Power_series}
\mathcal{M}_{\C}\llbracket t\rrbracket \,\widetilde{\to}\,\mathcal{M}_{\mathbb{N}}.
\ee

\begin{remark}
The `$\diamond$' decoration will also appear in ``preliminary'' versions of the power structure on $K_0(\Var_{\C})$ (Section \ref{sec:Power_structures}) and of the motivic exponential on $\mathcal M_S$ (Section \ref{sec:Motivic_Exp}). Just as in \cite{DavisonR}, in our formulas from Section \ref{Sec:Virtual_Motives} we need to prove that we are dealing with effective classes before removing the `$\diamond$' decoration and pass to the classical operations.
\end{remark}

\subsubsection{Power structures}\label{sec:Power_structures}

\begin{definition}[{\cite{GLMps}}]\label{def:power_structure}
A \emph{power structure} on a ring $R$ is a map 
\begin{align*}
(1+tR\llbracket t\rrbracket)\times R&\to 1+tR\llbracket t\rrbracket\\
(A(t),m)&\mapsto A(t)^m
\end{align*}
satisfying the following conditions: 
\begin{enumerate}
    \item $A(t)^0=1$,
    \item $A(t)^1=A(t)$, 
    \item $(A(t)\cdot B(t))^m=A(t)^m\cdot B(t)^m$, 
    \item $A(t)^{m+m'}=A(t)^m\cdot A(t)^{m'}$, 
    \item $A(t)^{mm'}=(A(t)^m)^{m'}$, 
    \item $(1+t)^m=1+mt+O(t^2)$, 
    \item $A(t)^m\big{|}_{t\to t^e}=A(t^e)^m$.
\end{enumerate}  
\end{definition}

Throughout we use the following:
\begin{notation}
Partitions $\alpha\vdash n$ are written as $\alpha=(1^{\alpha_1}\cdots i^{\alpha_i}\cdots s^{\alpha_s})$, meaning that there are $\alpha_i$ parts of size $i$. In particular we recover $n = \sum_ii\alpha_i$.
The \emph{automorphism group} of $\alpha$ is the product of symmetric groups $G_\alpha=\prod_i\mathfrak S_{\alpha_i}$. 
\end{notation}

If $X$ is a variety and $A(t)=1+\sum_{n>0}A_nt^n \in K_0(\Var_\C)\llbracket t\rrbracket$ is a power series, we define
\be\label{eqn:power_formula}
(A(t))_{\pr}^{[X]}=1+\sum_{n\geq 0}\sum_{\alpha\vdash n}\pi_{G_\alpha}\Biggl(\Biggl[\prod_i X^{\alpha_i}\setminus \Delta\Biggr]\cdot \prod_i A_i^{\otimes \alpha_i}\Biggr)t^{n}.
\ee
In the above formula, $\Delta\subset \prod_i X^{\alpha_i}$  is the ``big diagonal'' (the locus in the product where at least two entries are equal), and the product in big round brackets 
is a $G_\alpha$-equivariant motive, thanks to the power map \eqref{powermap}.
Gusein-Zade, Luengo and Melle-Hern{\'a}ndez have proved \cite[Thm.~2]{GLMps} that there is a unique power structure 
\[
(A(t),m) \mapsto A(t)^{m}
\]
on $K_0(\Var_{\C})$ for which the restriction to the case where all $A_i$ and $m$ are effective is given by the formula \eqref{eqn:power_formula}.
Since we always consider effective exponents when taking powers, we just recall the recipe for dealing with general $A(t)$ and effective exponent $[X]$. First, note that for any such $A(t)$ there is an effective $B(t)$ such that $A(t)\cdot B(t)=C(t)$ is effective. Then we have
\[
A(t)^{[X]} = (C(t))^{[X]}_{\pr}\cdot ((B(t))^{[X]}_{\pr})^{-1},
\]
where both factors in the right hand side are defined via \eqref{eqn:power_formula}.

As noted in \cite{BBS}, there is an extension of the power structure to $\mathcal{M}_{\C}$ uniquely determined by the substitution rules 
\[
A((-\mathbb{L}^{\frac{1}{2}})^nt)^{[X]}=A(t)^{[X]}\big|_{t\mapsto(-\mathbb{L}^{\frac{1}{2}})^nt}.
\]

\subsubsection{Motivic Exponential}\label{sec:Motivic_Exp}
The \emph{plethystic}, or \emph{motivic exponential} is a group isomorphism
\[
\Exp\colon t\mathcal{M}_\C\llbracket t\rrbracket\,\widetilde{\to}\,1+t\mathcal{M}_\C\llbracket t\rrbracket,
\]
converting sums into products.
First, define $\prExp=\sum_{n\geq 0}\prsigma^n$, where $\prsigma^n$ are (up to the identification \eqref{Iso_Power_series}) the lambda ring operations relative to the monoid $(\mathbb{N},+)$.  Then if $A$, $B \in \mathcal M^{\eff}_{\mathbb N}$ are effective classes, define
\[
\Exp(A-B)=\prExp(A)\cdot \prExp(B)^{-1}.
\]

If $(S, \nu\colon S\times S\rightarrow S)$ is a commutative monoid in the category of schemes, with a submonoid $S_+ \subset S$ such that the induced map $\coprod_{n\geq 1}S_+^{\times n}\rightarrow S$ is of finite type, we similarly define
\[
\prExp_{\nu}(A)=\sum_{n\geq 0} \prsigma^n_\nu(A),
\]
and for $A$, $B \in \mathcal M^{\eff}_S$ two effective classes, we set
\[
\Exp_{\nu}(A-B)=\prExp_{\nu}(A)\cdot \prExp_{\nu}(B)^{-1}.
\]

\subsubsection{Motives over symmetric products}\label{sec:symm_products_cup}

The machinery described so far will be applied in Section \ref{Sec:Virtual_Motives} to the following situation. For a variety $V$, we will consider $(\Sym(V),\cup)$, where
\[
\Sym(V) = \coprod_{n\geq 0}\Sym^n(V)
\]
can be viewed as a monoid via the morphism
\[
\Sym(V)\times\Sym(V)\xrightarrow{\cup} \Sym(V)
\]
sending two $0$-cycles on $V$ to their union. We consider the submonoid $\coprod_{n>0}\Sym^n(V)$ to construct the maps $\prExp_\cup$ and $\Exp_\cup$ as in Section \ref{sec:Motivic_Exp}.

In order to recover a formal power series in $\mathcal M_{\C}\llbracket t \rrbracket$ from a relative motive over $\Sym(V)$, we consider the operation 
\be\label{Power_Series_Sym}
\#_!\left(\sum_{n\geq 0}\,\bigl[M_n\to \Sym^n(V)\bigr]\right)=\sum_{n\geq 0}\, [M_n]t^n.
\ee
In other words we take the direct image along the ``tautological'' map $\#\colon \Sym(V)\rightarrow\mathbb{N}$ which collapses $\Sym^n(V)$ onto the point $n$. In the right hand side of \eqref{Power_Series_Sym}, we use \eqref{Iso_Power_series} to identify relative motivic weights over $\mathbb N$ and formal power series with coefficients in $\mathcal M_{\C}$.

\subsubsection{The virtual motive of a critical locus}

For a complex scheme $X$ of finite type over $\C$, recall the \emph{virtual Euler characteristic}
\[
\widetilde\chi(X) = \chi(X,\nu_X) = \sum_{n\in \Z}n\cdot \chi(\nu_X^{-1}(n)),
\]
where $\nu_X\colon X(\C)\to \Z$ is Behrend's canonical constructible function \cite{Beh}.

\begin{definition}[\cite{BBS}]\label{def:VM}
Let $X$ be a scheme. A motivic class $\xi \in \mathcal M_{\C}^{\hat\mu}$ such that $\chi(\xi) = \widetilde\chi(X)$ is called a \emph{virtual motive} for $X$. Here $\chi$ is the map \eqref{eulerchi}.
\end{definition}

\begin{definition}
A scheme $X$ is a \emph{critical locus} if there exists a smooth scheme $U$ and a regular function $f\colon U\to \A^1$ such that $X = Z(\dd f) \subset U$.
\end{definition}

A critical locus $X = Z(\dd f) \subset U$ does not only carry a canonical \emph{virtual fundamental class} $[X]^{\vir}\in A_0X$ (cf.~\cite{BFHilb}). It also supports a canonical relative motive
\[
\MF_{U,f} = \L^{-\frac{\dim U}{2}}\cdot \left[-\phi_f\right]_X \in \mathcal M_{X}^{\hat\mu},
\]
where $\MF$ stands for ``Milnor fibre'' and $[\phi_f]_X \in \mathcal M_{X}^{\hat\mu}$ is the (relative) motivic vanishing cycle class introduced by Denef and Loeser \cite{DenefLoeser1}.
It can be seen as the virtual motivic analogue of the sheaf of vanishing cycles $\Phi_f \in \DD^b_c(X)$.

\begin{notation}\label{notation:vir}
If $X=Z(\dd f) \subset U$ is a critical locus, $i\colon Z\into X$ is a subscheme and $X\to Y$ is a morphism, we define
\[
[Z\to Y]_{\vir} = (Z\into X \to Y)_!i^\ast \MF_{U,f} \in \mathcal M_{Y}^{\hat\mu}.
\]
When $X \to Y=\Spec \C$ is the structure morphism, we simply write $[Z]_{\vir}$.
\end{notation}

Set $[\phi_f] = \int_X [\phi_f]_X$. Then, by \cite[Prop.~2.16]{BBS}, the motivic weight
\[
[X]_{\vir} = \int_X \MF_{U,f} = \L^{-\frac{\dim U}{2}}\cdot \bigl[-\phi_f\bigr] \,\in\, \mathcal M_{\C}^{\hat\mu}
\]
is a virtual motive for $X$, in the sense of Definition \ref{def:VM}.

\begin{example}\label{ex:smooth_VM}
When $f=0$, we have $X=U$ and $[\phi_f] = -[X]$, thus 
\[
[X]_{\vir} = \L^{-\frac{\dim X}{2}}[X] \in \mathcal M_{\C}.
\]
This is the motivic analogue of the relation
\[
[X]^{\vir} = e(\Omega_X)\cap [X] \in A_0X.
\]
Assume $X$ is proper. Applying $\chi$ (resp.~the degree map) to the first (resp.~the second) identity yields the virtual Euler characteristic $\widetilde\chi(X) = (-1)^{\dim X}\chi(X)$.
\end{example}

\begin{example}
More generally, if $X=Z(\dd f)$ is proper, Behrend's theorem \cite{Beh} reads
\[
\int_{[X]^{\vir}} 1 = \chi [X]_{\vir} \,\in\, \Z,
\]
expressing the relation between the virtual class of $X$ and its virtual motive.
\end{example}

\section{Obstruction theories on Quot schemes}
For a coherent sheaf $F$ on a variety $Y$, and an integer $n\geq 0$, the Quot scheme 
\[
\Quot_Y(F,n)
\]
parameterises short exact sequences
\be\label{ses19}
0 \to S \to F \to Q \to 0,
\ee
where $Q$ is a sheaf supported in dimension $0$ with
\[
\chi(Q) = n.
\]
Throughout this section, $Y$ denotes a smooth complex projective $3$-fold, and $F$  a locally free sheaf (or vector bundle) of rank $r\geq 1$.

\subsection{Tangents and obstructions}
For a $0$-dimensional sheaf $Q$, one has $\Hom(F,Q)=\C^{r\cdot \chi(Q)}$. 
All other Ext groups vanish:
\be\label{eqn:vanishing}
\Ext^{3-i}(Q,F)^*\cong\Ext^i(F,Q\otimes \omega_Y) \cong\Ext^i(F,Q)=0,\quad i>0. 
\ee
Given a short exact sequence as in \eqref{ses19}, these vanishings induce isomorphisms
\be\label{Vanishing_Ext_2_3}
\Ext^i(F,S)\,\widetilde{\to}\,\Ext^i(F,F),\quad i=2,3.
\ee

\begin{lemma}\label{Lemma:start}
Let $Y$ be a smooth projective $3$-fold, $F$ a vector bundle on $Y$, and $F\onto Q$ a $0$-dimensional quotient  with kernel $S$. Then:
\begin{itemize}
    \item [(i)] if $F$ is simple, one has $\Hom(S,F) = \C$,
    \item [(ii)] if $F$ is rigid, one has $\Ext^1(S,F) = 0$.
\end{itemize} 
\end{lemma}

\begin{proof}
Applying $\Hom(-,F)$ to the exact sequence $S\into F\onto Q$ yields
\[
\Hom(Q,F) \to \Hom(F,F) \to \Hom(S,F) \to \Ext^1(Q,F),
\]
and the two outer groups vanish by \eqref{eqn:vanishing}, so if $F$ is simple we find
\be\label{Hom_S_into_F}
\C \,\widetilde{\to}\, \Hom(F,F) \,\widetilde{\to}\, \Hom(S,F),
\ee
proving (i).
The exact sequence above continues as
\[
\Ext^1(F,F) \to \Ext^1(S,F) \to \Ext^2(Q,F),
\]
where the rightmost group vanishes by \eqref{eqn:vanishing}, and the leftmost vanishes if $F$ is rigid (by definition), proving (ii).
\end{proof}

\begin{corollary}\label{IsoTg_InclusionOb}
In the situation of Lemma \ref{Lemma:start}, if $F$ is simple and rigid there is an isomorphism
\[
\Hom(S,Q) \,\widetilde{\to}\, \Ext^1(S,S),
\] 
and a linear inclusion
\[
\Ext^1(S,Q) \into \Ext^2(S,S).
\]
\end{corollary}

\begin{proof}
Applying $\Hom(S,-)$ to $S\into F\onto Q$ we obtain
\[
0\to \Hom(S,S) \xrightarrow{i} \Hom(S,F) \xrightarrow{u} \Hom(S,Q) \xrightarrow{\partial} \Ext^1(S,S) \to 0,
\]
where the $0$ on the right is $\Ext^1(S,F)=0$ obtained by Lemma \ref{Lemma:start}\,(ii). But we have a splitting $\Hom(S,S) = H^0(Y,\O_Y) \oplus \Hom(S,S)_0$, thus $i$ is an isomorphism since $\Hom(S,F)\cong \C$ by Lemma \ref{Lemma:start}\,(i). Hence $u=0$, which implies that $\partial$ is an isomorphism. Finally, the long exact sequence above continues as $0\to \Ext^1(S,Q) \to \Ext^2(S,S)$, proving the claim.
\end{proof}

Fix a short exact sequence 
as in \eqref{ses19},
defining a point $x = [F\onto Q] \in \Quot_Y(F)$. 
Consider the Quot functor $\mathsf{Quot}_Y(F)\colon \Sch_k^{\op}\to \Sets$ and let $\Art_{\C}\subset \Sch_{\C}^{\op}$ be the category of local Artinian $\C$-algebras (in other words $\Art_{\C}$ is the category of \emph{fat points}). As is well-known, the deformation functor 
\[
\Def_{x} = \Def_{F \onto Q}\subset \mathsf{Quot}_Y(F)\big|_{\Art_{\C}}
\]
defined by sending an algebra $A$ to the set of $A$-flat families of quotients restricting to $x$ over the closed fibre, is pro-representable and carries a tangent-obtruction theory $(T_1,T_2)$, in the sense of \cite{fga}, given by the vector spaces $T_i = \Ext^{i-1}(S,Q)$. However, this does not give rise to a perfect obstruction theory in the sense of Definition \ref{def:Ob_Theory}, for instance because higher Ext groups need not vanish. By Corollary \ref{IsoTg_InclusionOb}, the deformation theory of the quotients $F\onto Q$ is isomorphic to the deformation theory of the kernels $S\subset F$ --- see Lemma \ref{Prop:Def_Functors} for a precise statement. This allows us to modify the standard obstruction theory (essentially to get a larger obstruction space) by focusing on the \emph{kernels} of the surjections.

From now on in this section, we make the following:

\begin{assumption}\label{assumption193}
The locally free sheaf $F$ on the smooth projective $3$-fold $Y$ is simple and rigid. Moreover, either
\begin{itemize}
    \item [$(\star)$] $H^i(Y,\O_Y)=0$ for $i>0$ and $F$ is \emph{exceptional}, or
    \item [$(\ast)$] $Y$ is Calabi--Yau.
\end{itemize}
These are the assumptions of Theorem \ref{thmA}.
\end{assumption}

Recall that a simple coherent sheaf $F$ is exceptional if $\Ext^i(F,F)=0$ for all $i>0$. Note that, by our assumption, for any $S\in \Coh(Y)$ we have $\Ext^i(S,S)_0 = \Ext^i(S,S)$ for $i = 1,2$, and also for $i=3$ in case $(\star)$.

To get a \emph{perfect} obstruction theory, we will need the following vanishings.

\begin{prop}\label{Vanishing_Tracefree}
Let $(Y,F)$ satisfy Assumption \ref{assumption193}. Let $F\onto Q$ be a $0$-dimensional quotient with kernel $S$. Then
\[
\Hom(S,S)_0 = \Ext^3(S,S)_0 = 0.
\]
\end{prop}

\begin{proof}
From the splitting $\Hom(S,S) = H^0(Y,\O_Y)\oplus \Hom(S,S)_0$ induced by the trace, and the isomorphisms $\Hom(S,S) \isom \Hom(S,F) \isom \C$, we deduce $\Hom(S,S)_0 = 0$.
In the Calabi--Yau case $(\ast)$, by Serre duality we obtain the vanishing $\Ext^3(S,S)_0=0$. In case $(\star)$, consider the surjection $\Ext^3(F,S)\onto \Ext^3(S,S)$. By \eqref{Vanishing_Ext_2_3} we have $\Ext^3(F,S) = \Ext^3(F,F) = 0$, so $\Ext^3(S,S)=0$.
\end{proof}


\subsubsection{Virtual dimension and point-wise symmetry}\label{rmk:Dimension_Zero} 

In the perfect obstruction theory we want to build, the tangent space at $x=[S\into F\onto Q]\in \Quot_Y(F)$ is $\Ext^1(S,S) = \Hom(S,Q)$, and the obstruction space is $\Ext^2(S,S)$. Its virtual dimension at $x$ would then be
\be\label{virdim}
\vd_x = \ext^1(S,S) - \ext^2(S,S) = 1 - \chi(S,S) - \ext^3(S,S).
\ee
Note that $\chi(S,S) = \chi(F,F)$. In case $(\star)$, we have $\ext^3(S,S)=0$ and $\chi(F,F)=1$, therefore $\vd_x=0$. In the Calabi--Yau case, $\vd_x=0$ by Serre duality --- or, directly, because $\ext^3(S,S)=1$ and $\chi(F,F)=0$.
So the difference \eqref{virdim} is always zero. 

In fact, more is true: tangents are always dual to obstructions. This is clear in the Calabi--Yau case. In case $(\star)$, since $F$ is exceptional, one can use \emph{both} the vanishings $\Ext^2(F,S)=\Ext^3(F,S)=0$ from \eqref{Vanishing_Ext_2_3} to obtain an exact sequence
\[
0 \to \Ext^2(S,S) \to \Ext^3(Q,S) \to 0.
\]
Dualising, this is an isomorphism
\[
\Hom(S,Q) \,\widetilde{\to}\, \Ext^2(S,S)^*.
\]
To sum up, if we manage to produce a perfect obstruction theory with $\Ext^1(S,S)$, $\Ext^2(S,S)$ as tangents and obstructions, it will be $0$-dimensional and ``point-wise symmetric''. However, point-wise symmetry does not imply global symmetry (cf.~Definition \ref{def:Ob_Theory}), as shown by the case of $\Hilb^nY$ for a $3$-fold $Y$ that is not Calabi--Yau.

\subsection{Obstruction theory: construction}\label{Sec:Construct_Ob_Theory}
Let us shorten $\mathrm{Q}=\Quot_Y(F,n)$. Let $p\colon Y\times \mathrm{Q}\to \mathrm{Q}$ and $q\colon Y\times \mathrm{Q}\to Y$ be the projections. Consider the universal exact sequence
\[
0\to \SS\to q^*F \to \mathcal Q\to 0
\]
living over $Y\times \mathrm{Q}$.
The trace map 
\[
\tr_{\SS}\colon \RRlHom(\SS,\SS) \to \O_{Y\times \mathrm{Q}}
\]
has a canonical splitting, and we denote its kernel by
\[
\RRlHom(\SS,\SS)_0.
\]
The truncated cotangent complex $\mathbb{L}_{Y\times \mathrm{Q}}$ splits as $p^*\mathbb{L}_\mathrm{Q}\oplus q^*\mathbb{L}_Y$, so the \emph{truncated Atiyah class} (cf.~\cite[Def.~2.6]{HT}) 
\[
A(\SS) \in \Ext^1(\SS,\SS\otimes \mathbb{L}_{Y\times \mathrm{Q}})
\]
projects onto the factor
\begin{align*}
    \Ext^1(\SS,\SS\otimes p^*\mathbb{L}_\mathrm{Q}) &= \Ext^1(\SS^\vee\overset{\mathbf L}{\otimes}\SS,p^*\mathbb{L}_\mathrm{Q})\\
    &=\Ext^1(\mathbf{R}\lHom(\SS,\SS),p^*\mathbb{L}_\mathrm{Q}),
\end{align*}
which by the splitting of $\tr_{\SS}$ can be further projected onto
\[
\Ext^1(\mathbf{R}\lHom(\SS,\SS)_0,p^*\mathbb{L}_\mathrm{Q}).
\]
By Grothendieck duality along the smooth, proper $3$-dimensional morphism $p$, one has
\be\label{Grothendieck_Duality}
\RR p_* \RRlHom(\mathscr F,p^*\mathscr G\otimes \omega_p[3])=\RRlHom(\RR p_*\mathscr F,\mathscr G)
\ee
for $\mathscr F\in \DD^b(Y\times \mathrm{Q})$ and $\mathscr G\in \DD^b(\mathrm{Q})$,
where $\omega_p=q^*\omega_Y$ is the relative dualising sheaf.
Setting $\mathscr F = \RRlHom(\SS,\SS)_0\otimes \omega_p$ and $\mathscr G = \mathbb{L}_\mathrm{Q}$ in \eqref{Grothendieck_Duality}, we obtain
\[
    \RR p_* \RRlHom(\RRlHom(\SS,\SS)_0\otimes \omega_p,p^*\mathbb{L}_\mathrm{Q}\otimes \omega_p[3]) 
    =\RRlHom(\RR p_* (\RRlHom(\SS,\SS)_0\otimes \omega_p),\mathbb{L}_\mathrm{Q}),
\]
which after applying $h^{-2}\circ \RR\Gamma$ becomes
\begin{align*}
\Ext^1(\RRlHom(\SS,\SS)_0,p^*\mathbb{L}_\mathrm{Q}) &= \Ext^{-2}(\RR p_* (\RRlHom(\SS,\SS)_0\otimes \omega_p),\mathbb{L}_\mathrm{Q}) \\
&=\Hom(\mathbb E,\mathbb{L}_\mathrm{Q}),
\end{align*}
where we have set
\[
\mathbb E=\RR p_* (\RRlHom(\SS,\SS)_0\otimes \omega_p)[2].
\]
Under the above identifications, the truncated Atiyah class $A(\SS)$ determines a morphism
\[
\phi\colon \mathbb E\to \mathbb{L}_\mathrm{Q}.
\]
We can now give the proof of Theorem \ref{thmA}.

\begin{theorem}\label{THM:Ob_Theory}
If the pair $(Y,F)$ satisfies Assumption \ref{assumption193}, then $\phi$ is a perfect obstruction theory of virtual dimension $0$. If $Y$ is Calabi--Yau, it is symmetric.
\end{theorem}

\begin{proof}
The Quot scheme $\mathrm Q$ satisfies the assumptions stated in \cite[Section 4]{HT}, namely it is separated and it carries a universal simple sheaf. The latter is just the universal kernel $\SS \in \Coh(Y\times \mathrm Q)$ viewed as a $\mathrm Q$-flat family of simple sheaves on $Y$. Now the argument of \cite[Thm.~4.1]{HT} applied to $\SS$ proves that $\phi$ is an obstruction theory.

Let us shorten $\mathbb H = \RRlHom(\SS,\SS)_0$. Note that $\mathbb H$ is canonically self-dual.
The complex $\RR p_* \mathbb H$ is isomorphic in the derived category to a two-term complex of vector bundles $\mathcal T^\bullet = [\mathcal T^1\to \mathcal T^2]$ concentrated in degrees $1$ and $2$. More precisely, as in \cite[Lemma.~4.2]{HT}, the identification $\RR p_* \mathbb H = \mathcal T^\bullet$ follows from the vanishings
\[
\Ext^i(S,S)_0 = 0, \quad i\neq 1,2,
\]
that we proved in Proposition \ref{Vanishing_Tracefree}.
On the other hand, we have
\begin{align*}
(\RR p_* \mathbb H)^\vee[-1] &\,\,=\,\, \RRlHom(\RR p_* \mathbb H,\O_{\mathrm Q})[-1] \\
&\,\,=\,\, \RR p_\ast \RRlHom(\mathbb H,\omega_p[3])[-1] & \small{\textrm{Grothendieck duality}} \\
&\,\,=\,\, \RR p_\ast \RRlHom(\mathbb H,\omega_p)[2] & \small{\textrm{shift}} \\
&\,\,=\,\, \RR p_\ast \RRlHom(\mathbb H^\vee,\omega_p)[2] & \small{\mathbb H = \mathbb H^\vee} \\
&\,\,=\,\, \RR p_\ast (\mathbb H\otimes \omega_p)[2] \\
&\,\,=\,\, \mathbb E.
\end{align*}
Therefore $\mathbb E$ is perfect in $[-1,0]$, i.e.~$\phi$ is perfect.

For any point $x = [S\into F\onto Q]$, with inclusion $\iota_x\colon x \into \mathrm Q$, one has
\[
h^{i-1}(\mathbf L \iota_x^\ast \mathbb E^\vee) = \Ext^i(S,S),\quad i=1,2.
\]
Therefore we have $\vd = \rk \mathbb E= \ext^1(S,S) - \ext^2(S,S) = 0$, as observed in Section \ref{rmk:Dimension_Zero}. 

Let us prove symmetry in the Calabi--Yau case. The argument is standard --- see for instance \cite{BFHilb} --- but we repeat it here for completeness. Any trivialisation $\omega_Y \,\widetilde{\to}\,\O_Y$ induces, by pullback along $Y\times \mathrm Q \to Y$, a trivialisation $\omega_p \,\widetilde{\to}\,\O_{Y\times \mathrm Q}$, that we can use to construct an isomorphism
\[
\mathbb E[-2] \,\widetilde{\to}\,\RR p_* \mathbb H.
\]
 Dualising and shifting the last isomorphism, we get
\[
\theta\colon (\RR p_* \mathbb H)^\vee[-1]\,\widetilde{\to}\,\mathbb E^\vee[1],
\]
where the source is canonically identified with $\mathbb E$.
The symmetry condition $\theta^\vee[1] = \theta$ follows from \cite[Lemma 1.23]{BFHilb}.
\end{proof}

\begin{corollary}\label{Cor:Virtual_Class}
Under the assumptions of Theorem \ref{THM:Ob_Theory}, the Quot scheme $\Quot_Y(F,n)$ has a $0$-dimensional virtual fundamental class
\[
\bigl[\Quot_Y(F,n)\bigr]^{\vir} \in A_0(\Quot_Y(F,n)).
\]
\end{corollary}

Since the Quot scheme is proper, we can define Donaldson--Thomas type invariants
\be\label{def:dt_quot}
\DDT_{F,n} = \int_{[\Quot_Y(F,n)]^{\vir}}1 \in \Z,
\ee
representing the virtual number of points of the Quot scheme. They will be discussed in Section \ref{sec:Higher_Rk_DT}.

\subsection{Relation with moduli of simple sheaves}
In the proof of Theorem \ref{THM:Ob_Theory} we viewed the scheme $\Quot_Y(F,n)$ as a fine moduli space of simple sheaves via the universal kernel $\SS\subset q^*F$. We now prove that $\Quot_Y(F,n)$ is indeed an open subscheme of the moduli space 
\[
M_{Y,n}
\]
of simple sheaves with Chern character
$v_n = \ch(F)-(0,0,0,n)$.

\smallbreak
We now recall a classical result from Deformation Theory, stated in the language of tangent-obstruction theories --- see e.g.~\cite[Ch.~6]{fga}.

\begin{lemma}\label{Prop:Def_Functors}
Let $\DD$, $\DD'$ be two pro-representable deformation functors and let $(T_1,T_2)$, $(T'_1,T'_2)$ be tangent-obstruction theories on them. Let $\eta\colon \DD\to \DD'$ be a morphism inducing an isomorphism $h\colon T_1 \,\widetilde{\to}\,T_1'$ and a linear embedding $T_2\into T_2'$. Then $\eta$ is an isomorphism.
\end{lemma}

\begin{proof}
See \cite[Remark 2.3.8]{Sernesi} and the surrounding discussion.
\end{proof}

\begin{prop}\label{prop:Open_Immersion}
Let $F$ be a simple rigid vector bundle on a smooth projective $3$-fold $Y$. Then there is an open immersion $\Psi_n\colon \Quot_Y(F,n) \to M_{Y,n}$.
\end{prop}

\begin{proof}
The map $\Psi_n$ takes a surjection to its kernel. This is clearly a morphism, since $\SS$ is flat over $\mathrm Q=\Quot_Y(F,n)$. It is injective on points (by definition of the Quot functor) and locally of finite type (because the Quot scheme is of finite type over $\C$). 

We now show that $\Psi_n$ is formally \'etale. Fix a point $x = [F\onto Q]\in \Quot_Y(F,n)$ with $S=\ker (F\onto Q)$ and let $s = \Psi_n(x) = [S] \in M_{Y,n}$.
Consider the deformation functors $\Def_{F\onto Q}$ and $\Def_S$ and their tangent-obstruction theories given respectively by $T_i = \Ext^{i-1}(S,Q)$ and $T'_i = \Ext^i(S,S)$ for $i=1,2$. The natural transformation 
\[
\eta\colon \Def_{F\onto Q}\to \Def_S
\]
taking a surjection to its kernel involves pro-representable functors (note that $\Def_S$ is pro-representable because $S$ is simple), and it induces an isomorphism on tangent spaces and an injection on obstruction spaces (cf.~Corollary \ref{IsoTg_InclusionOb}). Then Lemma \ref{Prop:Def_Functors} implies that $\eta$ is an isomorphism of deformation functors.
This implies formal \'etaleness of $\Psi_n$ by a direct application of the formal criterion. In a little more detail, consider a square zero extension $\iota\colon T\into \overline T$ of fat points, and a commutative diagram
\[
\begin{tikzcd}[row sep = large,column sep = large]
T\arrow[hook,swap]{d}{\iota}\arrow{r}{g} & \mathrm Q \arrow{d}{\Psi_n} \\
\overline T\arrow[swap]{r}{\overline g}\arrow[dotted]{ur}[description]{\alpha} & M_{Y,n}
\end{tikzcd}
\]
where $\alpha$ is the \emph{unique} extension we need to find. Using pro-representability of $\Def_{F\onto Q}$ and $\Def_S$, the condition that $\eta$ is a natural isomorphism translates into a commutative diagram
\[
\begin{tikzcd}[row sep = large]
\Hom_x(\overline T,\mathrm Q)\isoarrow{d}\arrow{r}{\circ\, \iota} & \Hom_x(T,\mathrm Q)\isoarrow{d} \\
\Hom_s(\overline T,M_{Y,n})\arrow{r}{\circ\, \iota} & \Hom_s(T,M_{Y,n})
\end{tikzcd}
\]
where the vertical isomorphisms (composition with $\Psi_n$) are precisely the isomorphisms $\eta_T$ and $\eta_{\overline T}$. Since $\overline g\in \Hom_s(\overline T,M_{Y,n})$ lifts to a morphism $\alpha\in \Hom_x(\overline T,\mathrm Q)$ and both $\alpha \circ \iota$ and $g$ map to $\Psi_n \circ g\in \Hom_s(T,M_{Y,n})$, they must be equal, for the vertical map on the right is also an isomorphism. Thus $\alpha$ is the required (clearly unique) lift, proving that $\Psi_n$ is formally \'etale.

Thus $\Psi_n$ is an injective \'etale morphism, i.e.~an open immersion.
\end{proof}

\subsection{Symmetry in case $(\star)$}

In this section we assume the pair $(Y,F)$ satisfies $(\star)$ and we show that the obstruction theory constructed in Theorem \ref{THM:Ob_Theory} in this case becomes symmetric after suitably shrinking the Quot scheme.

The Quot-to-Chow morphism (see \cite[Section 6]{Grothendieck_Quot} or \cite[Cor.~$7.15$]{Rydh1} for its construction)
\[
\sigma_Y\colon \Quot_Y(F,n) \to \Sym^nY
\]
takes a quotient $[F\onto Q]$ to the $0$-cycle determined by the set-theoretic support $\Supp(Q)\subset Y$, weighted by the length. For any open subscheme $U\subset Y$, the preimage of $\Sym^nU\subset \Sym^nY$ under $\sigma_Y$ gives an open subscheme
\[
\mathrm Q_U \subset \Quot_Y(F,n) = \mathrm Q
\]
isomorphic to $\Quot_U(F|_U,n)$. Note that such Quot scheme makes sense, even though $U$ is only quasi-projective, because the support of a family of $0$-dimensional quotients is always proper over the base.

We now consider the diagram
\[
\begin{tikzcd}[row sep = large, column sep = large]
U\times \mathrm Q_U \arrow[swap]{dr}{\pi}\arrow[hook]{r}{a} & Y\times \mathrm Q_U\MySymb{dr} \arrow{d}{\overline p}\arrow[hook]{r}{j} & Y\times \mathrm Q\arrow{d}{p} \\
& \mathrm Q_U\arrow[hook]{r}{i} & \mathrm Q
\end{tikzcd}
\]
and form the pullback
\be\label{Pullback_E}
i^\ast \mathbb E = \RR \overline{p}_\ast (j^\ast \RRlHom(\SS,\SS)_0\otimes \omega_{\overline p})[2],
\ee
where the identification follows from base change and by $\omega_{\overline p} = j^\ast \omega_p$.
Since the inclusions $i$, $j$ and $a$ are open, their pullbacks are underived. Since dualising sheaves are invertible, tensor products $(-)\otimes \omega$ are also underived.

 Let us introduce the notation
\begin{align*}
\mathbb H_Y &= j^\ast \RRlHom(\SS,\SS)_0\in \DD(Y\times \mathrm Q_U),\\
\mathbb H_U &= a^\ast \mathbb H_Y\in \DD(U\times \mathrm Q_U),\\
\mathbb E_U &= \RR\pi_\ast(\mathbb H_U\otimes\omega_\pi)[2]\in \DD(\mathrm Q_U). 
\end{align*}
Since $\omega_\pi=a^\ast\omega_{\overline p}$, we can write
\be\label{Definition_F}   
\mathbb E_U[-2] 
=\RR \overline{p}_\ast \RR a_\ast(a^\ast\mathbb H_Y\otimes a^\ast\omega_{\overline p}) 
=  \RR\overline{p}_\ast(\mathbb H_Y\otimes \omega_{\overline p}\overset{\mathbf L}{\otimes} \RR a_\ast \O_{U\times \mathrm{Q}_U}).
\ee
where the first identity follows from $\RR\pi_\ast = \RR \overline p_\ast\circ \RR a_\ast$ and the second one uses the projection formula along the open immersion $a$. 
From Equations \eqref{Pullback_E} and \eqref{Definition_F}, the canonical morphism $\O_{Y\times \mathrm Q_U}\to \RR a_\ast \O_{U\times \mathrm{Q}_U}$ induces a canonical \emph{isomorphism}
\[
\alpha\colon i^\ast \mathbb E \to \mathbb E_U.
\]
To see that $\alpha$ is an isomorphism, it is enough to observe that its cone vanishes. The cone $K$ of $\O_{Y\times \mathrm{Q}_U} \to \RR a_\ast \O_{U\times \mathrm{Q}_U}$ is supported on $Z = Y\times \mathrm{Q}_U\setminus U\times \mathrm{Q}_U$. Now we note that $\mathbb H_Y\otimes^{\mathbf L} K$, which in principle is supported on $Z$, vanishes. To see this, first set $\mathcal F = q^\ast F$. Then, since $\mathcal S|_{Z} = \mathcal F|_{Z}$ (because the support of the quotients is now constrained on $U$),  we have that $\RRlHom(\mathcal S,\mathcal S)_0|_Z = \RRlHom(\mathcal F,\mathcal F)_0|_Z$, but $\RRlHom(\mathcal F,\mathcal F)_0 = 0$ since $F$ is exceptional. Thus $\mathbb H_Y\otimes^{\mathbf L} K=0$, and this implies that the cone of $\alpha$ also vanishes.

Composing the inverse of $\alpha$ with the map
\[
i^\ast \mathbb E \xrightarrow{i^\ast \phi} i^\ast \mathbb L_{\mathrm{Q}} \,\widetilde{\to}\, \mathbb L_{\mathrm{Q}_U},
\]
we obtain a perfect obstruction theory
\[
\phi_U\colon \mathbb E_U \to \mathbb L_{\mathrm{Q}_U}.
\]

\begin{prop}\label{Lemma:Symmetry}
Let $U\subset Y$ be an open subscheme such that $\omega_U$ is trivial.
Then the map $\phi_U\colon \mathbb E_U\to \L_{\mathrm Q_U}$ is a symmetric perfect obstruction theory.
\end{prop}

\begin{proof}
Any choice of trivialisation $\omega_U \,\widetilde{\to}\, \O_U$ induces, by pullback along $U\times \mathrm Q_U \to U$, a trivialisation $\omega_\pi \,\widetilde{\to}\, \O_{U\times \mathrm Q_U}$, that we use to construct an isomorphism
\[
\mathbb E_U[-2] = \RR\pi_\ast(\mathbb H_U\otimes \omega_\pi) \,\widetilde{\to}\, \RR\pi_\ast\mathbb H_U.
\]
From now on the proof is similar to that of Theorem \ref{THM:Ob_Theory}, except that we cannot use Grothendieck duality for $\pi$, since it is not proper. Thus we include full details.

Dualising and shifting the last displayed isomorphisms, we obtain
\[
\theta_U\colon (\RR\pi_\ast\mathbb H_U)^\vee[-1] \,\widetilde{\to}\,\mathbb E_U^\vee[1].
\]
We need to show that $(\RR\pi_\ast\mathbb H_U)^\vee[-1] = \mathbb E_U$. Note that, again by the projection formula along $a$, one has
\be\label{RpiE}
\RR\pi_\ast\mathbb H_U = \RR\overline{p}_\ast(\mathbb H_Y \overset{\mathbf L}{\otimes} \RR a_\ast \O_{U\times \mathrm{Q}_U}),
\ee
and moreover both complexes $\mathbb H_Y$ and $\RR a_\ast \O_{U\times \mathrm{Q}_U}$ are canonically self-dual.
Then
\begin{align*}
    (\RR\pi_\ast\mathbb H_U)^\vee[-1] 
    &\,\,=\,\, \RRlHom(\RR\overline{p}_\ast(\mathbb H_Y \overset{\mathbf L}{\otimes} \RR a_\ast \O_{U\times \mathrm{Q}_U}),\O_{\mathrm Q_U})[-1] & \small{\textrm{by \eqref{RpiE}}} \\
    &\,\,=\,\,\RR \overline{p}_\ast\RRlHom(\mathbb H_Y \overset{\mathbf L}{\otimes} \RR a_\ast \O_{U\times \mathrm{Q}_U},\omega_{\overline p}[3])[-1] & \small{\textrm{Grothendieck duality}}\\
    &\,\,=\,\,\RR \overline{p}_\ast\RRlHom(\mathbb H_Y \overset{\mathbf L}{\otimes} \RR a_\ast \O_{U\times \mathrm{Q}_U},\omega_{\overline p})[2] & \small{\textrm{shift}}\\
    &\,\,=\,\,\RR \overline p_\ast (\mathbb H_Y^\vee \overset{\mathbf L}{\otimes} (\RR a_\ast \O_{U\times \mathrm{Q}_U})^\vee \otimes \omega_{\overline p})[2] &  \small{\textrm{Hom and tensor}}\\
    &\,\,=\,\,\RR \overline p_\ast (\mathbb H_Y \overset{\mathbf L}{\otimes}\RR a_\ast \O_{U\times \mathrm{Q}_U} \otimes\omega_{\overline p}   )[2] & \small{\textrm{self-duality}}\\
    &\,\,=\,\, \mathbb E_U & \small{\textrm{by \eqref{Definition_F}}}.
\end{align*}
The symmetry property $\theta_U^\vee[1] = \theta_U$ again follows from \cite[Lemma 1.23]{BFHilb}.
\end{proof}

\begin{example}
Taking $Y=\P^3$, $U = \A^3$, $F$ an exceptional bundle on $\P^3$ of rank $r$, we see that $\Quot_{\A^3}(\O^{\oplus r},n)$ carries a symmetric perfect obstruction theory. As far as we know, it might not be possible to construct exceptional bundles on $\P^3$ of any given rank. However, $\Quot_{\A^3}(\O^{\oplus r},n)$ does have a symmetric obstruction theory for every $r$. This follows directly from its description as a critical locus \cite[Thm.~2.6]{BR18}, that we recall in Section \ref{subsec:local_model}.
\end{example}

\begin{aside}
The problems of constructing exceptional bundles and proving their stability are classical in Algebraic Geometry. By the foundational work of Dr\'ezet and Le Potier, all exceptional bundles on $\P^2$ are stable \cite{MR816365}. By work of Zube \cite{MR1074787}, the same is true for any K3 surface with Picard group $\Z$. This fact is used in \emph{loc.~cit.}~to prove that any exceptional bundle on $\P^3$ is stable. Mir\'o-Roig and Soares \cite{MR2425712} prove that if $Y\subset \P^n$ is a smooth complete intersection $3$-fold of type $(d_1,\ldots,d_{n-3})$, with $d_1+\cdots+d_{n-3}\leq n$ and $n\geq 4$, then any exceptional bundle on $Y$ is stable.
\end{aside}

\subsection{The stable case}
Let $H$ be a polarisation on the $3$-fold $Y$, i.e.~an ample class in $H^2(Y,\mathbb Z)$. Assume $F$ is a  $\mu_H$-stable (and rigid) vector bundle. Then the open immersion $\Psi_n$ of Proposition \ref{prop:Open_Immersion} factors through an open immersion
\be\label{Open_Immersion}
\Phi_n\colon \Quot_Y(F,n) \into \mathcal M_H^{\st}(v_n),
\ee
where the target is the moduli space of $\mu_H$-stable sheaves with Chern character $v_n = \ch(F) - (0,0,0,n)$. 

\begin{remark}\label{Connected_Component}
The open immersion $\Phi_n$ is also closed. Indeed, $\mathcal M_H^{\st}(v_n)$ is a separated scheme, so by properness of the Quot scheme $\Phi_n$ is a proper morphism. But a proper open immersion is a closed immersion. Hence $\Phi_n$ is the inclusion of a union of connected components.
\end{remark}

\begin{remark}
The perfect obstruction theory on the moduli space $\mathcal M_H^{\st}(v_n)$ constructed by Thomas \cite[Cor.~3.39]{ThomasThesis} (in the case when there are no strictly $\mu_H$-semistable sheaves and $Y$ has an anticanonical section) pulls back via $\Phi_n$ to the one constructed in Theorem \ref{THM:Ob_Theory}. For instance, in the Calabi--Yau case, the condition
\[
\gcd(r,\ch_1(F)\cdot H^2) = 1
\]
implies that there are no strictly semistable sheaves. In fact, it implies the stronger statement that there exists a universal sheaf over $\mathcal M^{\st}_H(v_n)$, see \cite[Cor.~B.2]{BR18} for a proof.
\end{remark}

\begin{example}
If $F = \O_Y$, one has $\Quot_Y(\O_Y,n) = \Hilb^nY$ and $\mathcal M^{\st}(1,0,0,-n)$, independent of the polarisation, is the moduli space of ideal sheaves (we are using that the determinant is fixed, thanks to $H^1(Y,\O_Y)=0$). In this case the open immersion $\Phi_n$ of \eqref{Open_Immersion} is also surjective: this recovers the classical identification of $\Hilb^nY$ with the moduli space of torsion free sheaves of Chern character $(1,0,0,-n)$.
\end{example}

\section{Higher rank Donaldson--Thomas invariants}\label{sec:Higher_Rk_DT}

\subsection{Calabi--Yau 3-folds}
Let us recall from \cite{BR18} the following weighted Euler characteristic calculation.

\begin{theorem}[{\cite[Thm.~A]{BR18}}] \label{Thm:QuotVirtualChi}
Let $Y$ be a smooth quasi-projective $3$-fold, $F$ a locally free sheaf of rank $r$. Then
\[
\sum_{n\geq 0}\widetilde\chi(\Quot_Y(F,n)) q^n = \mathsf M((-1)^rq)^{r \chi(Y)},
\]
where $\mathsf M(q) = \prod_{m\geq 1}(1-q^m)^{-m}$ is the MacMahon function.
\end{theorem}

Let now $Y$ be a projective Calabi--Yau $3$-fold, $F$ a simple rigid vector bundle. Set
\[
\DDT_F(q) = \sum_{n\geq 0} \DDT_{F,n} q^n,
\]
where $\DDT_{F,n}$ is the degree of the virtual class constructed in Corollary \ref{Cor:Virtual_Class}, see \eqref{def:dt_quot}.

Theorem \ref{Thm:QuotVirtualChi} has the following immediate consequence.
\begin{corollary}\label{cor:DT_CY3}
If $F$ is a simple rigid vector bundle on a Calabi--Yau $3$-fold $Y$, then
\be\label{eqn:DT_Higher_rank}
\DDT_F(q) = \mathsf M((-1)^rq)^{r\chi(Y)}.
\ee
\end{corollary}

\begin{proof}
Since the Quot scheme is proper and the obstruction theory constructed in Theorem \ref{THM:Ob_Theory} is \emph{symmetric}, by Behrend's theorem we have
\be\label{Virtual_chi_Quot}
\DDT_{F,n}=\widetilde\chi(\Quot_Y(F,n)).
\ee
The result then follows from Theorem \ref{Thm:QuotVirtualChi}.
\end{proof}

\subsection{The stable case}
Classical Donaldson--Thomas theory is defined for the moduli space of stable sheaves $\mathcal M^{\st}_H(\alpha)$, where $\alpha \in H^\ast(Y,\Q)$ is a given Chern character.
If $F$ is a $\mu_H$-stable rigid vector bundle,  \eqref{Virtual_chi_Quot} computes the virtual enumerative contribution of the \emph{connected component} (cf.~Remark \ref{Connected_Component})
\[
\Quot_Y(F,n) \subset \mathcal M^{\st}_H(v_n).
\]
Therefore Equation \eqref{eqn:DT_Higher_rank} can be seen as an explicit example of (classical) higher rank DT invariants.

\begin{example}\label{Ex:ACM_Quintic}
Recall that a vector bundle $F$ on a hypersurface $Y\subset \P^{r+1}$ is arithmetically Cohen--Macaulay if $H^i(Y,F(k))=0$ for $0<i<r$ and for all $k\in \Z$. By a result of Chiantini and Madonna \cite[Thm.~1.3]{MR1984199}, every stable arithmetically Cohen--Macaulay rank $2$ bundle on a general quintic $Y\subset \P^4$ is rigid. Therefore, since $\chi(Y)=-200$, for any such $F$ Equation \eqref{eqn:DT_Higher_rank} yields
\[
\DDT_F(q) = \mathsf M(q)^{-400}.
\]
\end{example}

This discussion motivates the following:

\begin{problem}
Construct examples of stable rigid vector bundles on Calabi--Yau $3$-folds.
\end{problem}

\subsection{General $3$-folds}\label{sec:DT_general_threefold}
Let $Y$ be a smooth projective $3$-fold, $F$ a vector bundle of rank $r$. The numbers $\DDT_{F,n}$ and their generating function $\DDT_F(q)$ can be defined as in \eqref{def:dt_quot} whenever the virtual class is defined. In the rank $1$ case, one has
\[
\DDT_{\O_Y}(q) = \mathsf M(-q)^{\int_Yc_3(T_Y\otimes \omega_Y)}.
\]
 See \cite{MNOP2} for a proof in the toric case and \cite{LEPA,JLi} for a general proof. We propose the following conjecture.

 \begin{conjecture}\label{conj}
Let $Y$ be a smooth projective $3$-fold, $F$ a vector bundle of rank $r$ on $Y$ such that the series $\DDT_F(q)$ is defined. Then there is an identity
\[
\DDT_F(q) = \mathsf M((-1)^rq)^{r\int_Yc_3(T_Y\otimes \omega_Y)}.
\]
 \end{conjecture}

Conjecture \ref{conj} does not seem to trivially follow from the existing arguments in the rank $1$ case. Besides the rank $1$ case, the formula is true in the Calabi--Yau case, by \eqref{eqn:DT_Higher_rank}. We hope to get back to this question in the future.\footnote{Update: Conjecture \ref{conj} has recently been proven for a toric $3$-fold and an equivariant exceptional locally free sheaf $F$ by Fasola, Monavari and the author in \cite[Thm.~C]{FMR_K-DT}. The general case remains open.}

\section{The virtual motive of the Quot scheme}\label{Sec:Virtual_Motives}

Throughout this section, we drop all assumptions on $(Y,F)$ we had previously. We let $Y$ be an arbitrary smooth quasi-projective $3$-fold, $F$ a vector bundle of rank $r$, and we consider $\Quot_Y(F,n)$.
In this section we construct a \emph{virtual motive} for this Quot scheme, i.e.~a motivic weight
\[
\bigl[\Quot_Y(F,n)\bigr]_{\vir} \in \mathcal M_{\C}
\]
such that applying the map $\chi$ of \eqref{eulerchi} yields
\begin{align*}
    \chi \bigl[\Quot_Y(F,n)\bigr]_{\vir} &=\widetilde\chi(\Quot_Y(F,n)) \\
&= (-1)^{rn}\chi(\Quot_Y(F,n)),
\end{align*}
where the second equality is equivalent to Theorem \ref{Thm:QuotVirtualChi}.

\subsection{The local model}\label{subsec:local_model}
In this subsection we work on the local Calabi--Yau $3$-fold $Y=\A^3$. Fix $r\geq 1$ and $n\geq 0$. In \cite[Thm.~2.6]{BR18}, it was proved that 
\[
\Quot_{\A^3}(\O^{\oplus r},n)
\]
is a critical locus. In the case $r=1$ (corresponding to the Hilbert scheme of points) this was already known \cite[Prop.~3.1]{BBS}. In particular, $\Quot_{\A^3}(\O^{\oplus r},n)$ carries both the structures (symmetric obstruction theory, virtual motive) recalled in Section \ref{Sec:Virtual_Stuff}.

\subsubsection{The critical structure on the local Quot scheme}
We briefly review from \cite{BR18} the critical structure on $\Quot_{\A^3}(\O^{\oplus r},n)$. The affine space
\[
\mathcal R = \Set{(A,B,C,v_1,\ldots,v_r)|A,B,C \in \End(\C^n),\,v_i \in \C^n},
\]
parameterising triples of $n$ by $n$ matrices and $r$-tuples of $n$-vectors, has dimension $3n^2+rn$. It can be seen as the space of $(n,1)$-dimensional representations of the $3$-loop quiver endowed with $r$ \emph{framings} issuing from an additional vertex $\infty$, cf.~Figure \ref{L3quiver}.

\begin{figure}[ht]
\begin{tikzpicture}[>=stealth,->,shorten >=2pt,looseness=.5,auto]
  \matrix [matrix of math nodes,
           column sep={3cm,between origins},
           row sep={3cm,between origins},
           nodes={circle, draw, minimum size=7.5mm}]
{ 
|(A)| \infty & |(B)| 0 \\         
};
\tikzstyle{every node}=[font=\small\itshape]
\path[->] (B) edge [loop above] node {$A$} ()
              edge [loop right] node {$B$} ()
              edge [loop below] node {$C$} ();

\node [anchor=west,right] at (-0.2,0.1) {$\vdots$};
\node [anchor=west,right] at (-0.3,0.95) {$v_1$};              
\node [anchor=west,right] at (-0.3,-0.85) {$v_r$};              
\draw (A) to [bend left=25,looseness=1] (B) node [midway,above] {};
\draw (A) to [bend left=40,looseness=1] (B) node [midway] {};
\draw (A) to [bend right=35,looseness=1] (B) node [midway,below] {};
\end{tikzpicture}
\caption{The framed $3$-loop quiver.}\label{L3quiver}
\end{figure}
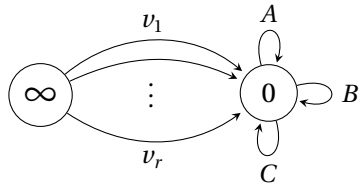

The group $\GL_n$ acts freely on the open subscheme
\[
U_{r,n}\subset \mathcal R
\]
parameterising tuples $(A,B,C,v_1,\ldots,v_r)$ such that the $\C$-linear span of the vectors of the form
\[
A^aB^bC^c\cdot v_i,\quad a,b,c \in \Z_{\geq 0},\quad 1\leq i\leq r
\]
has maximal dimension, i.e.~it equals $\C^n$. It was proved in \cite[Prop.~2.4]{BR18} that $U_{r,n}$ can be identified with a subspace of \emph{stable} framed representations of the $3$-loop quiver. The quotient
\[
\Quot^n_r = U_{r,n}/\GL_n
\]
is called \emph{non-commutative Quot scheme} in \cite{BR18}, by analogy with the case $r=1$, giving rise to the non-commutative Hilbert scheme. It is a smooth quasi-projective variety of dimension $2n^2+rn$. Consider the function
\be\label{Quiver_Potential}
f_{r,n}\colon \Quot^n_r \to \A^1,\quad (A,B,C,v_1,\ldots,v_r)\mapsto \Tr A[B,C].
\ee
\begin{theorem}[{\cite[Thm.~2.6]{BR18}}]\label{thm:quot_crit}
There is a scheme-theoretic isomorphism
\be\label{Critical_Quot}
\Quot_{\A^3}(\O^{\oplus r},n) \cong Z(\dd f_{r,n}) \subset \Quot^n_r.
\ee
\end{theorem}

\begin{example}
The potential $f_{r,1}$ vanishes for any $r$, so $\Quot_{\A^3}(\O^{\oplus r},1) = \Quot^1_r$ is smooth of dimension $r+2$. On the other hand, unlike the Hilbert scheme $\Hilb^n\A^3$, which is nonsingular for $n\leq 3$, the Quot scheme $\Quot_{\A^3}(\O^{\oplus r},r)$ is singular for all $r>1$. Indeed, the submodule $S=(x,y,z)^{\oplus r}\subset \C[x,y,z]^{\oplus r}$ defines a point whose tangent space has dimension $3r^2$. But $3r^2 = \dim \Quot^r_r > \dim \Quot_{\A^3}(\O^{\oplus r},r)$ since $f_{r,r}\neq 0$. Even in rank $1$, if we replace $\O_{\A^3}$ by the ideal sheaf of a line $L\subset \A^3$, the Quot scheme $\Quot_{\A^3}(\mathscr I_L,2)$ turns out to be singular, cf.~\cite[Example 2.7]{DavisonR}. 
\end{example}

The virtual motive induced by the critical structure \eqref{Critical_Quot} takes the form
\be\label{Virtual_Motive_Quot_Affine}
\left[\Quot_{\A^3}(\O^{\oplus r},n)\right]_{\vir} = \L^{-\frac{2n^2+rn}{2}}\cdot \left[-\phi_{f_{r,n}}\right]
\ee
and we shall see (cf.~Lemma \ref{Phi_No_Monodromy}) that it lives in the monodromy-free subring $\mathcal M_{\C}\subset \mathcal M^{\hat\mu}_{\C}$.
Let us form the generating function
\[
\mathsf Z_r(\A^3,t) = \sum_{n\geq 0}\, \left[\Quot_{\A^3}(\O^{\oplus r},n)\right]_{\vir}\cdot t^n\in \mathcal M_{\C}\llbracket t\rrbracket.
\]
The following computation was carried out following step by step the rank $1$ calculation by Behrend--Bryan--Szendr\H{o}i \cite{BBS}.

\begin{prop}[{\cite[Prop.~2.3.6]{ThesisR}}] \label{Prop:Motivic_Affine}
There is an identity
\be\label{id:mot}
\mathsf Z_{r}(\A^3,t) = \prod_{m=1}^\infty \prod_{k = 0}^{rm-1}\left(1-\L^{k+2-rm/2}t^m\right)^{-1}.
\ee
\end{prop}

A new proof of Proposition \ref{Prop:Motivic_Affine}, using wall-crossing for framed quiver representations, appeared recently in \cite{refinedDT_asymptotics}. 
We next compute (cf.~Corollary \ref{Prop:Motivic_Strata}) the motive \eqref{Virtual_Motive_Quot_Affine} and we show it is determined, via the power structure (cf.~Section \ref{sec:Power_structures}), by the virtual motivic contributions of the ``punctual strata'', just as in the rank $1$ case --- see \cite[Section 3]{BBS} and \cite[Section 3]{DavisonR}. This will allow us to define a virtual motive for all pairs $(Y,F)$ where $Y$ is a smooth quasi-projective $3$-fold and $F$ is a rank $r$ vector bundle on $Y$.

\subsubsection{The virtual motive of the Quot scheme of $\A^3$}
Let us fix $r\geq 1$ and for convenience let us shorten $\mathrm Q_{r,n} = \Quot_{\A^3}(\O^{\oplus r},n)$.
Consider the Quot-to-Chow morphism
\[
\sigma_n\colon \mathrm Q_{r,n} \to \Sym^n\A^3.
\]

\begin{lemma}\label{Phi_No_Monodromy}
The absolute motivic vanishing cycle $\phi_{f_{r,n}}$ satisfies the relation
\[
\left[\phi_{f_{r,n}}\right] = \left[f_{r,n}^{-1}(1)\right] - \left[f_{r,n}^{-1}(0)\right] \in \mathcal M_{\C},
\]
and the direct image along $\sigma_n$ of the motivic vanishing cycle is monodromy-free,
\[
\sigma_{n\,!\,}\left[\phi_{f_{r,n}}\right]_{\mathrm Q_{r,n}} \in \mathcal M_{\Sym^n\A^3} \subset \mathcal M^{\hat\mu}_{\Sym^n\A^3}.
\]
\end{lemma}

\begin{proof}
Let $\mathbb T = \G_m^3$ be the $3$-dimensional torus. The function $f_{r,n}$ is equivariant with respect to the primitive character $\chi(t) = t_1t_2t_3$, and a standard argument \cite{BBS} shows that the action of the diagonal subgroup $\G_m\subset \mathbb T$ is circle compact. Therefore the formula for $[\phi_{f_{r,n}}]$ follows from \cite[Thm.~B1]{BBS}. 

Let $L_3$ be the 3-loop quiver, i.e.~the quiver obtained from the one in Figure \ref{L3quiver} by removing all framings $\infty\to 0$. The function $(A,B,C) \mapsto \Tr A[B,C]$ on the space $\Rep_n(L_3)$ of $n$-dimensional representations of $L_3$ is reduced. This implies that $f_{r,n}^{-1}(0)\subset \Quot^n_r$ is a reduced hypersurface. Let $a\colon \mathrm Q_{r,n}\to Z$ be the affinisation of the Quot scheme. Then, again by \cite[Thm.~B1]{BBS}, the direct image $a_{\,!\,}[\phi_{f_{r,n}}]_{\mathrm Q_{r,n}}$ is monodromy-free. Since $\Sym^n\A^3$ is affine, $\sigma_n$ factors through $a$, thus $\sigma_{n\,!\,}[\phi_{f_{r,n}}]_{\mathrm Q_{r,n}}$ is also monodromy-free.
\end{proof}

The \emph{punctual Quot scheme} $\Quot_{\A^3}(\O^{\oplus r},n)_0\subset \mathrm Q_{r,n}$ is the locus of quotients $\O_{\A^3}^r \onto Q$ such that $Q$ is entirely supported at the origin $0\in \A^3$. It is the fibre of $\sigma_n$
over the point $n\cdot 0 \in \Sym^n_{(n)}\A^3$. We use the special notation
\[
\mathsf P_{r,n} = \left[\Quot_{\A^3}(\O^{\oplus r},n)_0\right]_{\vir}\in \mathcal M_{\C}
\]
for its virtual motivic contribution (see Notation \ref{notation:vir} for the definition of the right hand side), and we form the generating function
\[
\mathsf P_r(t) = \sum_{n\geq 0}\,\mathsf P_{r,n}\cdot t^n \in \mathcal M_{\C}\llbracket t\rrbracket.
\]
Define motivic weights
\[
\Omega_{r,n} \in \mathcal M_{\C}
\]
by the identity 
\be\label{Classes_Omega}
\mathsf P_r((-1)^rt) = \prExp \left(\sum_{n>0}\Omega_{r,n}\cdot t^n\right) \in \mathcal M_{\mathbb C}\llbracket t \rrbracket.
\ee

\begin{theorem}\label{Gen_Function_Affine}
There is an identity
\[
\mathsf Z_r(\A^3,(-1)^rt) = \mathsf P_r((-1)^rt)^{\L^3} \in \mathcal M_{\C}\llbracket t \rrbracket.
\]
\end{theorem}

\begin{proof}
The same analysis of \cite[Sec.~3]{DavisonR} shows that the relative virtual motives of $\mathrm Q_{r,n}$, viewed as relative classes over $\Sym(\A^3)$, are generated under $\Exp_\cup$ by the classes $\Omega_{r,n}$ defined in \eqref{Classes_Omega}, extended by the small diagonal. In other words, if $\Delta_n\colon \A^3\to \Sym^n\A^3$ denotes the small diagonal, a stratification argument combined with \cite[Prop.~1.12]{DavisonR} yields an identity
\be\label{Punctual_Exp}
\sum_{n \geq 0}\,(-1)^{rn}\left[\mathrm Q_{r,n} \xrightarrow{\sigma_n} \Sym^n\A^3\right]_{\vir} = \prExp_\cup\left(\sum_{n>0}  \Delta_{n\,!\,}\bigl(\Omega_{r,n}\boxtimes\bigl[\A^3 \xrightarrow{\id} \A^3\bigr]\bigr)\right)
\ee
in $\mathcal M_{\Sym(\A^3)}$.

Consider the map $\#\colon \Sym(\A^3) \to \mathbb N$ sending $\Sym^n\A^3$ to the point $n$. Its direct image $\#_!$ is described in \eqref{Power_Series_Sym}. By applying $\#_!$ to both sides of \eqref{Punctual_Exp}, and using \cite[Prop.~1.12]{DavisonR} along with Equation \eqref{Classes_Omega}, we deduce the identity
\be\label{diamond_1}
\mathsf Z_r(\A^3,(-1)^rt) = \mathsf P_r((-1)^rt)^{\L^3}_{\diamond}.
\ee
Next, we prove that $\Omega_{r,n}$ is effective for all $r$, $n$.
A straightforward calculation along the lines of \cite[Thm.~4.3]{BBS} allows us to verify, starting from \eqref{id:mot}, that
\[
\mathsf Z_r(\A^3,(-1)^rt) = 
\prExp\left(\frac{(-1)^rt\mathbb L^{\frac{3}{2}}}{\bigl(1-(-\mathbb L^{-\frac{1}{2}})^rt\bigr)\bigl(1-(-\mathbb L^{\frac{1}{2}})^rt\bigr)}\frac{\mathbb L^{-\frac{r}{2}}-\mathbb L^{\frac{r}{2}}}{\mathbb L^{-\frac{1}{2}}-\mathbb L^{\frac{1}{2}}}\right).
\]
By Equation \eqref{diamond_1}, this gives
\be\label{Punctual_Signed}
\mathsf{P}_r((-1)^rt) = \prExp\left(\frac{(-1)^rt\mathbb L^{-\frac{3}{2}}}{\bigl(1-(-\mathbb L^{-\frac{1}{2}})^rt\bigr)\bigl(1-(-\mathbb L^{\frac{1}{2}})^rt\bigr)}\bigl[\P^{r-1}\bigr]_{\vir}\right).
\ee
Combining \eqref{Classes_Omega} with the injectivity of $\prExp$ (see \cite[Lemma 1.11]{DavisonR}), an elementary comparison shows that
\be
\begin{split}\label{Omega_Explicit}
\Omega_{r,n} &= (-1)^{rn}\L^{-\frac{3}{2}}\cdot \L^{\frac{r(1-n)}{2}}\frac{\L^{rn}-1}{\L^r-1}\bigl[\P^{r-1}\bigr]_{\vir} \\
&=(-\L^{\frac{1}{2}})^{-rn-2} \bigl[\P^{rn-1}\bigr],
\end{split}
\ee
which belongs to $\mathcal M_{\C}^{\eff} \subset \mathcal M_{\C}$ for every $r$ and $n$. It follows that the classes $(-1)^{rn}\mathsf P_{r,n}$ are effective (because the plethystic exponential preserves effectiveness). This implies that 
\[
\mathsf P_r((-1)^rt)^{\L^3}_{\diamond}=\mathsf P_r((-1)^rt)^{\L^3},
\]
so the result follows from Equation \eqref{diamond_1}.
\end{proof}

\begin{remark}
Equation \eqref{Punctual_Exp} is the analogue of \cite[Thm.~3.17]{DavisonR}. The argument needed here is actually easier (and more similar to the setup of \cite{BBS}) than the one in \cite{DavisonR}. Indeed, in the present situation, there is only \emph{one} punctual contribution, whereas in \cite{DavisonR} two types of punctual contributions had to be considered.
\end{remark}

\begin{remark}
For $r=1$ we recover the effective classes
\[
\Omega_{1,n} = (-1)^n\L^{-\frac{3}{2}}\frac{\L^{\frac{n}{2}}-\L^{-\frac{n}{2}}}{\L^{\frac{1}{2}}-\L^{-\frac{1}{2}}}
\]
determining, via the identity  $\Exp\bigl(\sum_{n\geq 1}\Omega_{1,n}\cdot t^n\bigr) = \sum_{n\geq 0}\,[\Hilb^n(\A^3)_0]_{\vir}(-t)^n$, the virtual motives of the punctual Hilbert scheme of $\mathbb A^3$ defined in \cite{BBS}.
\end{remark}

\begin{remark}
Since the classes in Equation \eqref{Omega_Explicit} are effective, the `$\diamond$' decoration in Equation \eqref{Classes_Omega} can be removed, and we obtain the identity
\begin{align*}
    \mathsf P_r((-1)^rt) &= \Exp\left(\sum_{n>0}\Omega_{r,n}\cdot t^n\right)\\
&=\Exp \left(\sum_{n>0}(-\L^{\frac{1}{2}})^{-rn-2} \bigl[\P^{rn-1}\bigr]\cdot t^n\right).
\end{align*}
This relation can be viewed as the local \emph{motivic} analogue of the enumerative identity
\[
\mathsf M((-1)^rq)^{r\chi(Y)}=\exp\left(\sum_{n>0}(-1)^{rn-1}rn\cdot \mathrm N^Y_{n,0} q^n\right)
\]
where, for a projective $3$-fold $Y$, the number $\mathrm N^Y_{n,0} \in \Q$ is the virtual count of semistable sheaves $E \in \Coh(Y)$ supported in dimension $0$ and with $\chi(E)=n$.
\end{remark}

\begin{corollary}\label{Prop:Motivic_Strata}
For all $r\geq 1$ and $n\geq 0$, there is an identity 
\[
\left[\Quot_{\A^3}(\O^{\oplus r},n)\right]_{\vir} = \sum_{\alpha\vdash n}\pi_{G_\alpha}\left(\left[\prod_i\, (\A^3)^{\alpha_i}\setminus\Delta\right] \cdot \prod_i \mathsf P_{r,i}^{\alpha_i} \right) \in \mathcal M_{\C}.
\]
\end{corollary}

\begin{proof}
By the proof of Theorem \ref{Gen_Function_Affine}, the motives $(-1)^{ri}\mathsf P_{r,i}$ are effective. The result follows directly from the theorem and the power structure formula for an effective power series, cf.~\eqref{eqn:power_formula}.
\end{proof}

\subsection{Virtual motives for arbitrary $3$-folds}

Let $Y$ be a smooth quasi-projective $3$-fold, $F$ a vector bundle of rank $r$.

\begin{definition}\label{Def:Virtual_Motive_Any_Threefold}
We define the motivic weights $[\Quot_Y(F,n)]_{\vir} \in \mathcal M_\C$ by the identity
\[
\sum_{n\geq 0}\,\bigl[\Quot_Y(F,n)\bigr]_{\vir}((-1)^rt)^n = \mathsf P_r((-1)^rt)^{[Y]}.
\]
\end{definition}

Note that for $Y=\A^3$ this definition reconstruct the virtual motive of $\Quot_{\A^3}(\O^{\oplus r},n)$ by Theorem \ref{Gen_Function_Affine}.

Let us form the motivic partition function
\[
\mathsf Z_r(Y,t) = \sum_{n\geq 0}\,\bigl[\Quot_Y(F,n)\bigr]_{\vir}\cdot t^n \in \mathcal M_{\C}\llbracket t \rrbracket.
\]

\begin{lemma}
The motivic weight $[\Quot_Y(F,n)]_{\vir}$ is a virtual motive for $\Quot_Y(F,n)$.
\end{lemma}

\begin{proof}
The chain of equalities
\[
\chi \mathsf P_r((-1)^rt) = \chi \mathsf P_r((-1)^rt)^{\L^3} = \chi \mathsf Z_r(\A^3,(-1)^rt) = \mathsf M(t)^r
\]
implies that
\[
\chi \mathsf Z_r(Y,(-1)^rt) = \chi  \mathsf P_r((-1)^rt)^{[Y]} = (\chi \mathsf P_r((-1)^rt))^{\chi(Y)} = \mathsf M(t)^{r\chi(Y)}.
\]
The claim then follows by substituting $t\to (-1)^rt$ and comparing with Theorem \ref{Thm:QuotVirtualChi}.
\end{proof}

 We now derive a formula for $\mathsf Z_r(Y,(-1)^rt)$ in terms of the motivic exponential.


\begin{theorem}\label{Thm:Motivic_Partition_Function}
Let $Y$ be a smooth $3$-fold, $F$ a vector bundle of rank $r$. Then
\[
\mathsf Z_r(Y,(-1)^rt)=\Exp\left(
\frac{(-1)^rt\left[Y \times \P^{r-1}\right]_{\vir}}{\bigl(1-(-\mathbb L^{-\frac{1}{2}})^rt\bigr)\bigl(1-(-\mathbb L^{\frac{1}{2}})^rt\bigr) }\right).
\]
\end{theorem}

\begin{proof}
We have
\begin{align*}
\mathsf Z_r(Y,(-1)^rt) &= \mathsf P_r((-1)^rt)^{[Y]} \\
&= \Exp\left(\frac{(-1)^rt\mathbb L^{-\frac{3}{2}}}{\bigl(1-(-\mathbb L^{-\frac{1}{2}})^rt\bigr)\bigl(1-(-\mathbb L^{\frac{1}{2}})^rt\bigr)}\bigl[\P^{r-1}\bigr]_{\vir}\right)^{[Y]}
\end{align*}
where we used Formula \eqref{Punctual_Signed} in the second equality.
Recall from Example \ref{ex:smooth_VM} that $[Y]_{\vir}=\L^{-3/2}[Y]$.
Then the formula follows by $[Y]_{\vir} \cdot [\mathbb P^{r-1}]_{\vir} = [Y\times \mathbb P^{r-1}]_{\vir}$.
\end{proof}

\begin{remark}
The formula of Theorem \ref{Thm:Motivic_Partition_Function} can also be rewritten as
\[
\mathsf Z_r(Y,(-1)^rt)=
\Exp\left((-1)^rt\left[Y \times \P^{r-1}\right]_{\vir}\Exp\bigl((-\L^{-\frac{1}{2}})^rt+(-\L^{\frac{1}{2}})^rt\bigr)\right).
\]
\end{remark}

\begin{remark}
An abstract power structure formula for the naive motive of $\Quot_Y(F,n)$ was given in \cite{ricolfi2019motive}, generalising the work of Gusein-Zade, Luengo and Melle-Hern\'andez for the case of $\Hilb^nY$ \cite{GLMHilb}.
\end{remark}

\subsection{Motivic Donaldson--Thomas invariants}

Let $F$ be a simple rigid vector bundle on a Calabi--Yau $3$-fold $Y$. Then the motivic weight $[\Quot_Y(F,n)]_{\vir}$ of Definition \ref{Def:Virtual_Motive_Any_Threefold} can be seen as a (rank $r$) motivic Donaldson--Thomas invariant, for it refines the enumerative invariant $\DDT_{F,n}$ computed by \eqref{Virtual_chi_Quot}. Similarly, the motivic partition function $\mathsf Z_r(Y,t)$ of Theorem \ref{Thm:Motivic_Partition_Function} can be seen as a motivic refinement of the enumerative generating function $\DDT_F(q)$ computed in \eqref{eqn:DT_Higher_rank}.

An explicit example of such higher rank refinement (in the context of \emph{stable} sheaves) is provided by a rank $2$ arithmetically Cohen--Macaulay stable bundle $F$ on a generic quintic $3$-fold $Y$ in $\P^4$, cf.~Example \ref{Ex:ACM_Quintic}.

\subsubsection{An open question}
Let $(Y,H)$ be a polarised Calabi--Yau $3$-fold. The moduli space of stable sheaves $\mathcal M^{\st}_H(\alpha)$ with Chern character $\alpha$ has the structure of an oriented d-critical locus in the sense of \cite[Def.~2.31]{Joyce1}. See \cite{NO16} for a proof of existence of orientations, i.e.~square roots of the virtual canonical bundle. Let $F$ be a stable rigid vector bundle on $Y$. Then the connected component
\[
\Phi_n\colon \Quot_Y(F,n) \into \mathcal M^{\st}_H(v_n)
\]
inherits an oriented d-critical structure. In particular, by the results of \cite{BJM}, each orientation $L$ induces a canonical virtual motive
\[
\MF^L_{\Quot_Y(F,n)} \in \mathcal M_{\C}^{\hat\mu}.
\]
It is an interesting question to check whether there exists an orientation $L$ such that the induced virtual motives $\MF^L_{\Quot_Y(F,n)}$ agree with the ones defined in this paper (cf.~Definition \ref{Def:Virtual_Motive_Any_Threefold}). As far as we know, this is still unknown in the rank $1$ case \cite{BBS}, i.e.~for the Hilbert scheme of points.

\subsection{Vanishing cycle cohomology}\label{sec:VC_Cohomology}
It follows from \cite[Thm.~6.3]{Davison16} and the description of the space $\Quot_{\A^3}(\O^{\oplus r},n)$ as a fine moduli space of quiver representations \cite[Prop.~2.4]{BR18}, that the mixed Hodge structure on the total compactly supported \emph{vanishing cycle cohomology}
\[
\mathrm H_c\left(\Quot_{\A^3}(\O^{\oplus r},n),\Phi_{f_{r,n}}\right)
\]
is pure of Tate type for all $n$. 
Here $f_{r,n}\colon \Quot^n_r\to \A^1$ is the regular function defined in \eqref{Quiver_Potential} and 
\[
\Phi_{f_{r,n}}=\varphi_{f_{r,n}}\Q_{\Quot^n_r}
\]
is the perverse sheaf of vanishing cycles. Just as in \cite[Example 6.4]{Davison16}, thanks to purity we can evaluate the Hodge polynomial of the Quot scheme starting from the identity
\be\label{Quot_Motivic}
\sum_{n\geq 0}\,\bigl[\Quot_{\A^3}(\O^{\oplus r},n)\bigr]_{\vir}\cdot \bigl(\mathbb L^{-\frac{r}{2}}t\bigr)^n = \prod_{m=1}^\infty \prod_{k = 0}^{rm-1}\left(1-\L^{k+2-rm}t^m\right)^{-1},
\ee
obtained by applying the renormalisation $t\mapsto \L^{-r/2}t$ to Equation \eqref{id:mot}. According to \cite[Section 1.1]{Davison16}, the Hodge polynomial (more precisely, the Hodge series) of a cohomologically graded mixed Hodge structure $\mathscr L$ is the formal power series
\[
\mathsf h(\mathscr L;x,y,z) = \sum_{a,b,c\in \Z}\dim(\Gr^b_H(\Gr^W_{b+c}(\mathrm H^a(\mathscr L))))x^by^cz^a.
\]
The E-series is given by $\mathsf E(\mathscr L;x,y)=\mathsf h(\mathscr L;x,y,-1)$, and the weight series is defined by the further specialisation
\[
\mathsf w(\mathscr L;q^{1/2})=\mathsf E(\mathscr L;q^{\frac{1}{2}},q^{\frac{1}{2}}),
\]
where $q$ keeps track of cohomological degree.
We have
\[
\mathsf h\bigl(\mathrm H_c\bigl(\Quot_{\A^3}(\O^{\oplus r},n),\Phi_{f_{r,n}}\bigr);x,y,z\bigr) = \mathsf w\left(\mathrm H_c\left(\Quot_{\A^3}(\O^{\oplus r},n),\Phi_{f_{r,n}}\right);q\right)
\]
after the substitution $q^2 = xyz^2$. Thus, by specialising $\L \to q^2$, we deduce from \eqref{Quot_Motivic} the identity
\begin{multline}\label{VC_calculation}
\sum_{n\geq 0}\mathsf h\left(\mathrm H_c\bigl(\Quot_{\A^3}(\O^{\oplus r},n),\Phi_{f_{r,n}}\bigr);x,y,z\right)\cdot (xyz^2)^{-n^2-rn}\cdot t^n = \\ \prod_{m = 1}^\infty \prod_{k=0}^{rm-1}\left(1-(xyz^2)^{k+2-rm}t^m\right)^{-1}.    
\end{multline}

\bibliographystyle{amsplain-nodash}
\bibliography{bib}

\end{document}

%% file: macros.tex


\def\be{\begin{equation}}    
\def\ee{\end{equation}}
\def\bitem{\begin{itemize}}
\def\eitem{\end{itemize}}
\def\benum{\begin{enumerate}}
\def\eenum{\end{enumerate}}

\def\into{\hookrightarrow}
\def\onto{\twoheadrightarrow}

\def\op{\textrm{op}}
\def\isom{\cong}

\def\Var{\mathrm{Var}}
\def\Sch{\mathrm{Sch}}
\def\Sets{\mathrm{Sets}}
\def\Def{\mathsf{Def}}

\def\st{\mathrm{st}}

\def\L{\mathbb L}
\def\A{\mathbb A}

\def\C{\mathbb C}

\def\P{\mathbb P}
\def\Q{\mathbb Q}
\def\G{\mathbb G}
\def\L{\mathbb{L}}
\def\SS{\mathcal S}
\def\RR{\mathbf R}

\def\Z{\mathbb Z}
\def\ext{\mathrm{ext}}

\def\O{\mathscr O}
\def\DDT{\mathsf{DT}}

\def\vd{\mathrm{vd}}

\def\MF{\mathsf{MF}}

\newcommand{\prsigma}{{}^{\diamond}\!\!\sigma}
\newcommand{\prExp}{{}^{\diamond}\!\!\Exp}
\DeclareMathOperator{\Quot}{Quot}
\DeclareMathOperator{\DD}{D}
\DeclareMathOperator{\Hilb}{Hilb}
\DeclareMathOperator{\eff}{eff}
\DeclareMathOperator{\Gr}{Gr}

\DeclareMathOperator{\ch}{ch}

\DeclareMathOperator{\tr}{tr}
\DeclareMathOperator{\id}{id}

\DeclareMathOperator{\Rep}{Rep}

\DeclareMathOperator{\vir}{vir}
\DeclareMathOperator{\QCoh}{QCoh}
\DeclareMathOperator{\Coh}{Coh}

\DeclareMathOperator{\Spec}{Spec\,}

\DeclareMathOperator{\Supp}{Supp\,}

\DeclareMathOperator{\GL}{GL}

\DeclareMathOperator{\dd}{d}
\DeclareMathOperator{\Tr}{Tr}

\DeclareMathOperator{\Sym}{Sym}

\DeclareMathOperator{\Ext}{Ext}

\DeclareMathOperator{\Hom}{Hom}
\DeclareMathOperator{\lHom}{{\mathscr Hom}}
\DeclareMathOperator{\RRlHom}{{\mathbf{R}\mathscr Hom}}

\DeclareMathOperator{\End}{End}

\DeclareMathOperator{\Exp}{Exp}
\DeclareMathOperator{\rk}{rk}

\DeclareMathOperator{\pr}{\diamond}
\DeclareMathOperator{\Art}{Art}

\theoremstyle{definition}

\newtheorem*{lemma*}{Lemma}
\newtheorem*{theorem*}{Theorem}
\newtheorem*{example*}{Example}
\newtheorem*{fact*}{Fact}
\newtheorem*{notation*}{Notation}
\newtheorem*{definition*}{Definition}
\newtheorem*{prop*}{Proposition}
\newtheorem*{remark*}{Remark}
\newtheorem*{corollary*}{Corollary}
\newtheorem*{conventions*}{Conventions}

\newtheorem{definition}{Definition}[section]
\newtheorem{example}[definition]{Example}

\newtheorem{aside}[definition]{Aside}

\newtheorem{notation}[definition]{Notation}
\newtheorem{remark}[definition]{Remark}
\newtheorem{conjecture}[definition]{Conjecture}

\newtheoremstyle{thm} 
        {3mm}
        {3mm}
        {\slshape}
        {0mm}
        {\bfseries}
        {.}
        {1mm}
        {}
\theoremstyle{thm}
\newtheorem{theorem}[definition]{Theorem}
\newtheorem{corollary}[definition]{Corollary}
\newtheorem{lemma}[definition]{Lemma}
\newtheorem{prop}[definition]{Proposition}
\newtheorem{thm}{Theorem}

\newtheoremstyle{ass} 
        {3mm}
        {3mm}
        {\rmfamily}
        {0mm}
        {\scshape}
        {.}
        {1mm}
        {}
\theoremstyle{ass}
\newtheorem{assumption}[definition]{Assumption}
\newtheorem{problem}[definition]{Problem}

\usepackage{tikz}
\usepackage{tikz-cd}
\usepackage{rotating}
\newcommand*{\isoarrow}[1]{\arrow[#1,"\rotatebox{90}{\(\sim\)}"
]}
\usetikzlibrary{matrix,shapes,arrows,decorations.pathmorphing}
\tikzset{commutative diagrams/arrow style=math font}
\tikzset{commutative diagrams/.cd,
mysymbol/.style={start anchor=center,end anchor=center,draw=none}}
\newcommand\MySymb[2][\square]{%
  \arrow[mysymbol]{#2}[description]{#1}}
\tikzset{
shift up/.style={
to path={([yshift=#1]\tikztostart.east) -- ([yshift=#1]\tikztotarget.west) \tikztonodes}
}
}

\DeclareMathAlphabet{\mathpzc}{OT1}{pzc}{m}{it}